\documentclass[12pt]{amsart}
\usepackage{amssymb,amsmath,amsfonts,amsthm,mathrsfs,fullpage}
\usepackage{color}

\usepackage{lipsum} 
\usepackage{color}

\numberwithin{equation}{section}

\renewcommand{\AA}{\mathbb A}

\newcommand{\FF}{\mathbb F}
\newcommand{\GG}{\mathbb G}

\newcommand{\PP}{\mathbb P}
\newcommand{\QQ}{\mathbb Q}

\newcommand{\ZZ}{\mathbb Z} 

\newcommand{\Zhat}{\widehat\ZZ}

\newcommand{\calB}{\mathcal B}
\newcommand{\calC}{\mathcal C}

\newcommand{\OO}{\mathcal O}
\newcommand{\calP}{\mathcal P}
\newcommand{\calU}{\mathcal U}

\newcommand{\calR}{\mathcal R}
\newcommand{\calS}{\mathcal S}
\newcommand{\calT}{\mathcal T}

\newcommand{\hH}{\mathfrak h}
\newcommand{\gG}{\mathfrak g}

\newcommand{\m}{\mathfrak m}
\newcommand{\aA}{\mathfrak a}

\newcommand{\q}{\mathfrak q}
\newcommand{\p}{\mathfrak p}
\renewcommand{\P}{\mathfrak P}

\def\Spec{\operatorname{Spec}} 

 \def\Gal{\operatorname{Gal}}
\def\End{\operatorname{End}}

\def\Spec{\operatorname{Spec}}

\def\tr{\operatorname{tr}}

\def\Gal{\operatorname{Gal}}

\def\sl{\mathfrak{sl}}
\def\gl{\mathfrak{gl}}
  
\def \GL{\operatorname{GL}}  

\def \SL{\operatorname{SL}}

\def\Aut{\operatorname{Aut}} 
\def\End{\operatorname{End}}
\def\Frob{\operatorname{Frob}}

\def\tr{\operatorname{tr}}

 \def \Aut {\operatorname{Aut}}

\definecolor{purple}{rgb}{1,0,1}

\def\bbar#1{\setbox0=\hbox{$#1$}\dimen0=.2\ht0 \kern\dimen0 
\overline{\kern-\dimen0 #1}}
\newcommand{\Qbar}{{\overline{\mathbb Q}}} 
\newcommand{\Kbar}{\bbar{K}} 
 
\newcommand{\FFbar}{\overline{\FF}}

\def\sep{{\operatorname{sep}}}
\def\un{{\operatorname{un}}}

\newcommand{\defi}[1]{\textsf{#1}} 
\newtheorem{thm}{Theorem}[section]
\newtheorem{lemma}[thm]{Lemma}

\newtheorem{prop}[thm]{Proposition}

\theoremstyle{definition}

\theoremstyle{remark}
\newtheorem{remark}[thm]{Remark}

\newenvironment{romanenum}{\hfill \begin{enumerate} }{\end{enumerate}}
\newenvironment{alphenum}{\hfill \begin{enumerate} }{\end{enumerate}}

\definecolor{webcolor}{rgb}{0.8,0,0.2}
\definecolor{webbrown}{rgb}{.6,0,0}
\usepackage[
        colorlinks,
        linkcolor=webbrown,  filecolor=webcolor,  citecolor=webbrown, 
        backref,
        pdfauthor={David Zywina}, 
        pdftitle={Drinfeld modules with maximal Galois action},
]{hyperref}
\usepackage[alphabetic,backrefs,lite]{amsrefs} 

\begin{document}
\title{Drinfeld modules with maximal Galois action}
\author{David Zywina}
\address{Department of Mathematics, Cornell University, Ithaca, NY 14853, USA}
\email{zywina@math.cornell.edu}

\begin{abstract}

With a fixed prime power $q>1$, define the ring of polynomials $A=\FF_q[t]$ and its fraction field $F=\FF_q(t)$.   For each pair $a=(a_1,a_2) \in A^2$ with $a_2$ nonzero, let $\phi(a)\colon A\to F\{\tau\}$ be the Drinfeld $A$-module of rank $2$ satisfying $t\mapsto t+a_1\tau+a_2\tau^2$.    The Galois action on the torsion of $\phi(a)$ gives rise to a Galois representation $\rho_{\phi(a)}\colon \Gal(F^\sep/F)\to \GL_2(\widehat{A})$, where $\widehat{A}$ is the profinite completion of $A$.   We show that the image of $\rho_{\phi(a)}$ is large for random $a$.  More precisely, for all $a\in A^2$ away from a set of density $0$, we prove that the index $[\GL_2(\widehat{A}):\rho_{\phi(a)}(\Gal(F^\sep/F))]$ divides $q-1$ when $q>2$ and divides $4$ when $q=2$.   We also show that the representation $\rho_{\phi(a)}$ is surjective for a positive density set of $a\in A^2$.
\end{abstract}

\maketitle

\section{Introduction}

Throughout we fix a finite field $\FF_q$ with $q$ elements.  Define the polynomial ring $A=\FF_q[t]$ and its fraction field $F=\FF_q(t)$.

\subsection{Background}
We now recall some notions concerning Drinfeld modules.  For an introduction see~\cite{MR1423131,MR902591,MR0384707}.   
Let $K$ be an \defi{$A$-field}, i.e., a field $K$ with a fixed ring homomorphism $\iota\colon A\to K$.  Using $\iota$, we can view $K$ as a field extension of $\FF_q$.  

Let $K\{\tau\}$ be the ring of skew polynomials over $K$, i.e., the ring of polynomials in the indeterminate $\tau$ with coefficients in $K$ that satisfy the commutation rule $\tau c = c^q \tau$ for all $c\in K$.  We can identify $K\{\tau\}$ with a subring of $\End(\GG_{a,\, K})$ by identifying $\tau$ with the Frobenius map $X\mapsto X^q$.    Let $\partial_0 \colon K\{\tau\} \to K$ be the ring homomorphism $\sum_{i} a_i \tau^i \mapsto a_0$.   

A \defi{Drinfeld $A$-module} over $K$ is a ring homomorphism 
\[
\phi\colon A \to K\{\tau\}, \quad a \mapsto \phi_a
\]
such that $\partial_0\circ \phi=\iota$ and $\phi(A)\not\subseteq K$.    The \defi{characteristic} of $\phi$ is the kernel $\p_0$ of $\iota$; equivalently, the kernel of $\partial_0 \circ \phi \colon A \to K$.  If $\p_0=(0)$, then we say that $\phi$ has \defi{generic characteristic} and we may use $\iota$ to view $K$ as a field extension of $F$.  The Drinfeld module $\phi$ is determined by 
$\phi_t = \sum_{i=0}^r a_i \tau^i$ where we have $a_i\in K$ with $a_r\neq 0$; the positive integer $r$ is called the \defi{rank} of $\phi$.  

Fix a separable closure $K^\sep$ of $K$. The Drinfeld module $\phi$ endows $K^\sep$ with an $A$-module structure.  More precisely, $a\cdot x := \phi_a(x)$ for $a\in A$ and $x\in K^\sep$, where we are using our identification of $K\{\tau\}$ with a subring of $\End(\GG_{a,K})$.   We shall write ${}^\phi \!K^\sep$ if we wish to emphasize $K^\sep$ with this particular $A$-module structure.  For a nonzero ideal $\aA$ of $A$, the \defi{$\aA$-torsion} of $\phi$ is the $A$-module
\[
\phi[\aA] :=\{x\in {}^\phi\!K^\sep : a \cdot x = 0 \text{ for all } a\in \aA \} =  \{x \in K^\sep :  \phi_a(x) = 0 \text{ for all } a\in \aA \}.
\]

Suppose that $\aA$ is relatively prime to the characteristic $\p_0$.  Then $\phi[\aA]$ is a free $A/\aA$-module of rank $r$.    The absolute Galois group $\Gal_K:=\Gal(K^\sep/K)$ acts on $\phi[\aA]$ and respects the $A$-module structure.  This action can be expressed in terms of a Galois representation
\[
\bbar\rho_{\phi, \aA} \colon \Gal_K \to \Aut(\phi[\aA]) \cong \GL_r(A/\aA).
\]

For the rest of the section, assume that $\phi$ has generic characteristic. 
By choosing bases compatibly and taking the inverse limit, we obtain a single representation 
\[
\rho_\phi\colon \Gal_K \to \GL_r(\widehat{A})
\]
that encodes the Galois action on the torsion submodule of ${}^\phi \!K^\sep$, where $\widehat{A}$ is the profinite completion of $A$.    The representation $\rho_\phi$ is continuous when the groups are endowed with their profinite topologies. 

For a nonzero prime ideal $\lambda$ of $A$, let $\rho_{\phi,\lambda}\colon \Gal_K\to \GL_r(A_\lambda)$ be the representation obtained by composing $\rho_\phi$ with the quotient map $\GL_r(\widehat{A})\to \GL_r(A_\lambda)$, where $A_\lambda$ is the inverse limit of the rings $A/\lambda^i$  with $i\geq 1$.  The representation $\rho_{\phi,\lambda}$ encodes the Galois action on the $\lambda$-power torsion of $\phi$.   We can identify $\rho_\phi$ with $\prod_\lambda \rho_{\phi,\lambda}$ by using the natural isomorphism $\GL_r(\widehat{A})=\prod_\lambda\GL_r(A_\lambda)$, where the product is over the nonzero prime ideals of $A$.

Pink and R\"utsche \cite{MR2499412} have described the image of $\rho_\phi$ up to commensurability when $K$ is finitely generated.  For simplicity, we only state the version for which $\phi$ has no extra endomorphisms.  Recall that the ring $\End_{\Kbar}(\phi)$ of \defi{endomorphisms} is the centralizer of $\phi(A)$ in $\Kbar\{\tau\}$, where $\Kbar\supseteq K^\sep$ is an algebraic closure of $K$.

\begin{thm}[Pink-R\"utsche] \label{T:MT}
Let $\phi$ be a Drinfeld $A$-module of rank $r$ over a finitely generated field $K$.   Assume that $\phi$ has generic characteristic and that $\End_{\Kbar}(\phi)=\phi(A)$.   Then $\rho_\phi(\Gal_K)$ is an open subgroup of $\GL_r(\widehat{A})$.  Equivalently, $\rho_\phi(\Gal_K)$ has finite index in $\GL_r(\widehat{A})$.   
\end{thm}

Theorem~\ref{T:MT}, especially with $r\geq 2$, is a Drinfeld module analogue of Serre's open image theorem for non-CM elliptic curves, cf.~\cite{MR0387283}.  

In this article, we are interested in producing Galois representations $\rho_\phi$ with largest possible image.   We shall focus our attention on the most immediate case which is $r=2$ and $K=F$.   We shall show that there are infinitely many nonisomorphic Drinfeld modules $\phi$ over $F$ of rank $2$ with $\rho_\phi(\Gal_F)=\GL_2(\widehat{A})$.

\subsection{Density} \label{SS:density}

For a fixed integer $n\geq 1$, we will want to talk about properties holding for ``most'' $a\in A^n$.  To make this precise, we introduce the notion of density.  For any subset $S\subseteq A^n$ and positive integer $d$, we let $S(d)$ be the set of $(a_1,\ldots, a_n) \in S$ with $\deg(a_i) \leq d$ for all $1\leq i \leq n$.   
Define
\[
\overline{\delta}(S):=\limsup_{d\to +\infty} \frac{|S(d)|}{|A^n(d)|} \quad \text{ and } \quad \underline{\delta}(S):=\liminf_{d\to +\infty} \frac{|S(d)|}{|A^n(d)|};
\]
these are the \defi{upper density} and \defi{lower density} of $S$, respectively.  Note that $|A^n(d)|=q^{n(d+1)}$.  When $\overline\delta(S)=\underline\delta(S)$, we call the common value the \defi{density} of $S$ and denote it by $\delta(S)$.  Of course, $\delta(A^n)=1$.   
\subsection{Main result}

We shall always view $F$ as an $A$-field via the inclusion $A\subseteq F$.  For each pair $a=(a_1,a_2) \in A^2$ with $a_2\neq 0$,  let 
\[
\phi(a)\colon A\to F\{\tau\}, \quad \alpha\mapsto \phi(a)_\alpha
\] 
be the Drinfeld $A$-module over $F$ for which $\phi(a)_t = t + a_1\tau + a_2\tau^2$.   The Drinfeld module $\phi$ has rank $2$ and generic characteristic.  Associated to $\phi$, we have a Galois representation $\rho_{\phi(a)} \colon \Gal_F \to \GL_2(\widehat{A})$ that is uniquely determined up to isomorphism.  

We now define the following sets which consist of pairs $a\in A^2$ for which $\rho_{\phi(a)}$ has especially large image:
\begin{itemize}
\item
Let $S_1$ be the set of $a\in A^2$ with $a_2\neq 0$ for which $\rho_{\phi(a)}(\Gal_F)= \GL_2(\widehat{A})$.
\item
When $q\neq 2$, let $S_2$ be the set of $a\in A^2$ with $a_2\neq 0$ for which $\rho_{\phi(a)}(\Gal_F)\supseteq \SL_2(\widehat{A})$ and $[\GL_2(\widehat{A}): \rho_{\phi(a)}(\Gal_F)]$ divides $q-1$.
\item
When $q= 2$, let $S_2$ be the set of $a\in A^2$ with $a_2\neq 0$ for which $\rho_{\phi(a)}(\Gal_F)$ contains the commutator subgroup of $\GL_2(\widehat{A})$ and $[\GL_2(\widehat{A}): \rho_{\phi(a)}(\Gal_F)]$ divides $4$.
\item 
Let $S_3$ be the set of $a\in A^2$ with $a_2\neq 0$ for which $\rho_{\phi(a),\lambda}(\Gal_F)= \GL_2(A_\lambda)$ for all nonzero prime ideals $\lambda$ of $A$.  
\end{itemize}

Our main theorem shows that the sets $S_1$, $S_2$ and $S_3$ are large.

\begin{thm} \label{T:main new} 
\begin{romanenum}
\item \label{T:main new i}
There is a subset of $S_1$ with positive density.
\item \label{T:main new ii}
The set $S_2$ has density $1$.
\item \label{T:main new iii}
The set $S_3$ has density $1$.
\end{romanenum}
\end{thm}

Loosely, Theorem~\ref{T:main new}(\ref{T:main new ii}) says that for a ``randomly chosen'' $a\in A^2$ the index of $\rho_{\phi(a)}(\Gal_F)$ in $\GL_2(\widehat{A})$ is finite and divides $q-1$ or $4$ when $q\neq 2$ or $q=2$, respectively.  Theorem~\ref{T:main new}(\ref{T:main new i})  shows that $\rho_{\phi(a)}(\Gal_F)= \GL_2(\widehat{A})$ holds for many $a\in A^2$.

\begin{remark}
\begin{romanenum}
\item
Assume $q\neq 2$.  Take any $(a_1,a_2)\in A^2$ with $a_2$ monic and $\deg(a_2)\equiv 1 \pmod{q-1}$.   We have $[\widehat{A}^\times:\det(\rho_{\phi(a)}(\Gal_F))] = q-1$ and hence $[\GL_2(\widehat{A}):\rho_{\phi(a)}(\Gal_F)] \geq q-1$, cf.~Theorem~\ref{T:det image}.   This shows that the set $S_1$ does not have density $1$ and that the integer $q-1$ occurring in the definition of $S_2$ is optimal for Theorem~\ref{T:main new}(\ref{T:main new ii}) to hold.
\item
The commutator subgroup of $\GL_2(\widehat{A})$ is $\SL_2(\widehat{A})$ when $q\neq 2$.  A group theoretic complication that arises when $q=2$ is that the commutator subgroup of $\GL_2(\widehat{A})$ is a proper subgroup of $\SL_2(\widehat{A})$; in fact, it is a subgroup of index $4$.  This is the underlying reason why definition of $S_2$ is different when $q=2$.  
\item
Theorem~\ref{T:main new}(\ref{T:main new i}) gives counterexamples to \cite[Theorem~4.4]{MR4415990} which would imply that $S_1=\emptyset$ when $q=2$.  There seem to be issues when working with wildly ramified quadratic extensions of $F$ in their proof.  Also, Theorem~\ref{T:main new}(\ref{T:main new i}) gives counterexamples to \cite[Theorem~5.4]{MR4415990} which would imply that $S_1$ has density $0$ when $q=3$.
\end{romanenum}
\end{remark}

\subsection{Explicit examples}

For each prime power $q>1$, we also give an example of a rank $2$ Drinfeld module whose Galois representation is surjective.

\begin{thm} \label{T:example}
 Let $\phi\colon A\to F\{\tau\}$ be the Drinfeld module for which 
\[
\phi_t = 
\begin{cases}
t + \tau -t^{q-1} \tau^2 & \text{ if $q\neq 2$,}\\
t+t^3\tau +(t^2+t+1)\tau^2 & \text{ if $q=2$.}
\end{cases}
\]
Then $\rho_\phi(\Gal_F)=\GL_2(\widehat{A})$.   
\end{thm}

\subsection{Overview}

We give a brief overview of the paper.    Consider a Drinfeld module $\phi\colon A\to F\{\tau\}$ of rank $2$ with $\End_{\bbar{F}}(\phi)=\phi(A)$.   

In \S\ref{S:group theory lambda-adic}, we give a criterion that will allow us to show that a subgroup of $\GL_2(A_\lambda)$ is equal to the full group.  In \S\ref{S:group theory adelic}, we give a criterion that will allow us to show that a subgroup of $\GL_2(\widehat{A})$ contains the commutator subgroup of $\GL_2(\widehat{A})$.     We will try to apply these results to the subgroup $\rho_\phi(\Gal_F)$ of $\GL_2(\widehat{A})$.   It is these group theoretic results that motivate the structure of the paper.

Galois representations for Drinfeld modules defined over local fields will be studied in \S\ref{S:inertia}.   This will be used in our proofs to understand the action of inertia subgroups on the torsion of $\phi$ at primes for which our Drinfeld modules have semistable reduction.   In particular, this will give a way to construct subgroups of $\bbar\rho_{\phi,\aA}(\Gal_F)$ that we have some control over.

In \S\ref{S:Frobenius polynomials}, we recall that the representations $\bbar\rho_{\phi,\aA}$ are compatible and give rise to Frobenius polynomials.  These polynomials have coefficients in $A$ and are computable.  In \S\ref{S:determinant}, we recall a theorem of Gekeler that will give an explicit expression for the index of $\det(\rho_\phi(\Gal_F))$ in $\widehat{A}^\times$.   

An important step in showing that $\rho_\phi$ has large image is to prove that the representations $\bbar\rho_{\phi,\lambda}\colon \Gal_F\to \GL_2(\FF_\lambda)$ are irreducible for all nonzero prime ideals $\lambda$.  In \S\ref{S:irreducibility}, we prove that this holds for all but finitely many $\lambda$ and give an explicit bound on the norms of any possible exceptions.

In \S\ref{S:Hilbert irreducibility}, we prove a version of Hilbert's irreducibility theorem.   We use it to show that for a fixed nonzero ideal $\aA$ of $A$, we have $\rho_{\phi(a),\aA}(\Gal_F)=\GL_2(A/\aA)$ for all $a\in A^2$ away from a set of density $0$ (for future reference, we will give a version that holds for arbitrary rank $r\geq 2$).    The set of density $0$ will depend on $\aA$, so Hilbert's irreducibility theorem cannot be used by itself to prove our main theorems.

In \S\ref{S:sieving}, we use all the above ingredients and some careful sieving to get information on the image of $\rho_{\phi(a)}$ for all $a\in A^2$ away from a set of density $0$.   In particular, for all $a\in A^2$ away from a set of density $0$, we show that $\rho_{\phi(a),\lambda}(\Gal_F)=\GL_2(A_\lambda)$ for all nonzero prime ideals $\lambda$ of $A$, and also show that $\rho_{\phi(a)}(\Gal_F)$ and $\GL_2(\widehat{A})$ have the same commutator subgroup.  The proof of Theorem~\ref{T:main new} in the case $q\neq 2$ will then be quickly proved in 
\S\ref{SS:main new proof -main}.

Suppose that $q=2$.  In \S\ref{S:wild}, we give a condition on $\phi$ that ensures that the homomorphism $\Gal_F \to \GL_2(\widehat{A})/[\GL_2(\widehat{A}),\GL_2(\widehat{A})]$ obtained by composing $\rho_\phi$ with the quotient map is surjective.   This is achieved by considering the ramification at the place $\infty$ of $F$.  In \S\ref{SS:main new proof -special}, we prove the remaining case of Theorem~\ref{T:main new}.

Finally, \S\ref{S:proof of examples} is dedicated to the computation of the Galois images of the explicit Drinfeld modules from Theorem~\ref{T:example}.

\subsection{Some earlier results}

In the unpublished preprint \cite{1110.4365} the author proved Theorem~\ref{T:example} when $q\geq 5$ is odd.   This was extended to $q=3$ and $q=2^e\geq 4$ in \cite{MR4415990}.   The original goal of this work was to reprove this in a manner that could readily generalize to most Drinfeld modules like as in Theorem~\ref{T:main new}.   When $q=p^e$ with $p\geq 5$ and $p\equiv 1 \pmod{3}$, Chen gave an example of a rank $3$ Drinfeld $A$-module $\phi\colon A\to F\{\tau\}$ for which $\rho_\phi(\Gal_F)=\GL_3(\widehat{A})$, cf.~\cite{MR4410021}.  There are also some recent papers proving Hilbert irreducibility like results, cf.~\cite{MR4755197, 2407.14264,2403.15109}. 

\subsubsection{Elliptic curves} Let us briefly mention the analogous case of elliptic curves over a fixed number field $K$.   Consider an elliptic curve $E$ over $K$.  For each integer $n\geq 1$, the Galois action on the $n$-torsion points of $E$ gives rise to a continuous representation $\bbar\rho_{E,n}\colon \Gal_K:=\Gal(\Kbar/K)\to \GL_2(\ZZ/n\ZZ)$.   By choosing compatible bases, these combine to give a single Galois representation $\rho_E\colon \Gal_K\to \GL_2(\Zhat)$.

Serre observed that for an elliptic curve over $\QQ$, we can never have $\rho_E(\Gal_\QQ)\supseteq \SL_2(\Zhat)$, cf.~\cite[Prop.~22]{MR0387283}.  One ingredient of this obstruction is that $\det\circ \rho_E \colon \Gal_\QQ\to \Zhat^\times$ is the cyclotomic character and hence the fixed field in $\Qbar$ of its kernel is the maximal abelian extension of $\QQ$ by the Kronecker--Weber theorem.   When $K\neq \QQ$ there is no such obstruction and we have $\rho_E(\Gal_K)\supseteq \SL_2(\Zhat)$ for a ``random'' elliptic curve $E$ over $K$, cf.~\cite{1110.4365}.   Jones \cite{MR2563740} prove that we have $[\GL_2(\Zhat):\rho_E(\Gal_\QQ)]=2$ for a ``random'' elliptic curve $E/\QQ$.

\subsection{Notation}

For a nonzero ideal $\aA$ of $A$, we let $A_\aA$ be the inverse limit of the rings $A/\aA^i$ with $i\geq 1$.   We have natural isomorphisms
\[
A_\aA = \prod_{\p\supseteq \aA} A_\p \quad \text{ and }\quad \widehat{A}=\prod_\p A_\p,
\]
where the product is over the nonzero prime ideals $\p$ of $A$.  Each ring $A_\p$ is a complete discrete valuation ring.   

Consider a nonzero prime ideal $\p$ of $A$.   Define the residue field $\FF_\p:=A/\p$ and denote its cardinality by $N(\p)$.   We let $\deg(\p)$ be the degree of the field extension $\FF_\p/\FF_q$.  We have $N(\p)=q^{\deg(\p)}$ and $\deg(\p)$ is also the degree of any polynomial $\pi\in A$ with $\p=(\pi)$.   Let $F_\p$ be the completion of $F$ at $\p$; it is a local field with valuation ring $A_\p$.  Let $v_\p\colon F_\p^\times \to \ZZ$ be the corresponding valuation normalized so that $v_\p(F_\p^\times)=\ZZ$ and we set $v_\p(0)=+\infty$.

For a field $K$, let $K^\sep$ be a separable closure of $K$ and define the absolute Galois group $\Gal_K=\Gal(K^\sep/K)$.

Let $\phi\colon A\to K\{\tau\}$ be a Drinfeld module of rank $2$.  The \defi{$j$-invariant} of $\phi$ is $j_\phi:=a_1^{q+1}/a_2 \in K$, where $\phi_t=t+a_1\tau+a_2\tau^2$.

\section{Group theoretic criterion for large $\lambda$-adic image} \label{S:group theory lambda-adic}

Throughout this section, we fix a finite field $\FF$ and define the ring of formal power series $R:=\FF[\![\pi]\!]$.  The ring $R$ is a complete discrete valuation ring with maximal ideal $\p$ generated by $\pi$ and has residue field $\FF$.   

The following proposition gives a criterion to check if a subgroup of $\GL_2(R)$ is actually the full group.  
Note that for a nonzero prime ideal $\lambda$ of $A=\FF_q[t]$, the ring $A_\lambda$ is of the form $\FF_\lambda[\![\pi]\!]$, cf.~\cite[Chapter II \S4 Theorem~2]{MR450380}.  In particular, Proposition~\ref{P:full GL2R criterion} gives a group theoretic criterion to check if $\rho_{\phi,\lambda}(\Gal_F)$ is equal to the full group $\GL_2(A_\lambda)$ for a Drinfeld $A$-module $\phi\colon A\to F\{\tau\}$.

\begin{prop} \label{P:full GL2R criterion}
Let $G$ be a closed subgroup of $\GL_2(R)$ that satisfies the following conditions:
\begin{alphenum}
\item \label{P:full GL2R criterion a}
$\det(G)=R^\times$,
\item \label{P:full GL2R criterion b}
the image of $G$ modulo $\p$ is $\GL_2(\FF)$, 
\item \label{P:full GL2R criterion c}
if $|\FF|>2$, then there is an element $I+\pi B$ of $G$ with $B\in M_2(R)$ so that $B$ modulo $\p$ is a nonscalar matrix in $M_2(\FF)$,
\item \label{P:full GL2R criterion d}
if $|\FF|=2$, then the image of $G$ modulo $\p^2$ is $\GL_2(R/\p^2)$,
\item \label{P:full GL2R criterion e}
if $|\FF|=2$, then $G \cap \SL_2(R)$ contains an element whose reduction modulo $\p$ in $\SL_2(\FF)$ has order $2$.
\end{alphenum}
Then $G=\GL_2(R)$.  
\end{prop}

We will prove the proposition in \S\ref{SS:proof of full GL2R criterion}.  When $|\FF|>3$, Proposition~\ref{P:full GL2R criterion} also follows from \cite[Proposition~4.1]{MR2499412}.   

\subsection{Groups over a finite field}

\begin{prop} \label{P:group contains SL2 condition}
Let $G$ be a subgroup of $\GL_2(\FF)$ that acts irreducibly on $\FF^2$ and contains a subgroup of cardinality $|\FF|$.  Then $G\supseteq\SL_2(\FF)$.
\end{prop}
\begin{proof} 
Let $P_1$ be a subgroup of $G$ of order $|\FF|$; it is a $p$-Sylow subgroup of $\GL_2(\FF)$, where $p$ is the characteristic of $\FF$.   There is a unique $1$-dimensional $\FF$-subspace $W_1$ of $\FF^2$ that is fixed by every element of $P_1$.  If $P_1$ is a normal subgroup of $G$, then $W_1$ would be stable under the action of $G$ which would contradict our irreducibility assumption.  Therefore, there is a second subgroup $P_2\neq P_1$ of $G$ with cardinality $|\FF|$.  Let $W_2$ be the unique $1$-dimensional $\FF$-subspace of $\FF^2$ that is fixed by every element of $P_2$.  We have $W_1\neq W_2$.   After conjugating $G$ in $\GL_2(\FF)$, we may assume that $(1,0)\in W_1$ and $(0,1)\in W_2$, and hence 
\[
P_1=\left\{ \left(\begin{smallmatrix}1 & x \\0 & 1\end{smallmatrix}\right) : x \in \FF \right\} \quad \text{ and }\quad 
P_2=\left\{ \left(\begin{smallmatrix}1 &0 \\x & 1\end{smallmatrix}\right) : x \in \FF \right\} .
\]
Now take any matrix $M=\left(\begin{smallmatrix}A & B \\C & D\end{smallmatrix}\right) \in \SL_2(\FF)$.   First suppose that $B\neq 0$.  For $a,b,c\in \FF$, we have
\[
\left(\begin{smallmatrix}1 & 0 \\a & 1\end{smallmatrix}\right)
\left(\begin{smallmatrix}1 & b \\0 & 1\end{smallmatrix}\right)
\left(\begin{smallmatrix}1 & 0 \\c & 1\end{smallmatrix}\right)
=\left(\begin{smallmatrix}1+bc & b \\a+c+abc & 1+ab\end{smallmatrix}\right).
\]
So setting $b=B$ and solving $1+bc=A$ and $1+ab=D$ for $a$ and $c$ (recall that $B\neq0$), we find an expression for $M$ as a product of matrices in $P_1$ and $P_2$ (that $a+c+abc=C$ is automatic since our matrices have determinant $1$ and $b=B\neq 0$).  Therefore $M\in G$.   An analogous argument shows that $M\in G$ when $C\neq 0$.   Finally in the case $B=C=0$, we simply note that 
$\left(\begin{smallmatrix}-1 & 0 \\0 & -1\end{smallmatrix}\right)=\left(\begin{smallmatrix}1 & 0 \\1 & 1\end{smallmatrix}\right)\left(\begin{smallmatrix}1 & -2 \\0 & 1\end{smallmatrix}\right)\left(\begin{smallmatrix}1 & 0 \\1 & 1\end{smallmatrix}\right)\left(\begin{smallmatrix}1 & -2 \\0 & 1\end{smallmatrix}\right) \in G.$
\end{proof}

\begin{lemma} \label{L:finite field facts}
\begin{romanenum}
\item  \label{L:finite field facts i}
If $|\FF|>3$, then the group $\SL_2(\FF)/\{\pm I\}$ is nonabelian and simple.
\item \label{L:finite field facts ii}
If $|\FF|>3$, then $[\SL_2(\FF),\SL_2(\FF)]=\SL_2(\FF)$, i.e., $\SL_2(\FF)$ is perfect.
\item \label{L:finite field facts iii}
If $|\FF|>2$, then $[\GL_2(\FF),\GL_2(\FF)]=\SL_2(\FF)$.
\end{romanenum}
\end{lemma}
\begin{proof}
Parts (\ref{L:finite field facts i}) and (\ref{L:finite field facts ii}) are shown in \cite[\S3.3.2]{MR2562037}.  Part (\ref{L:finite field facts iii}) follows from (\ref{L:finite field facts ii}) when $|\FF|>3$ and can be checked directly when $|\FF|=3$.
\end{proof}

We define $\gl_2(\FF):=M_2(\FF)$ and we let $\sl_2(\FF)$ be the subgroup consisting of matrices with trace $0$.  These $\FF$-vector spaces are Lie algebras under the pairing $[x,y]=xy-yx$.

\begin{lemma} \label{L:PR adjoint}
If $|\FF|>2$, then any subgroup of $\gl_2(\FF)$ that is invariant under conjugation by $\GL_2(\FF)$ either contains $\sl_2(\FF)$ or consists only of scalar matrices. 
\end{lemma}
\begin{proof}
This is Proposition~2.1 of \cite{MR2499412} when $|\FF|\geq 4$.  A direct computation shows that this also holds when $|\FF|=3$.
\end{proof}

\subsection{Filtration of a closed subgroup} \label{SS:filtration}

Consider a closed subgroup $G$ of $\GL_2(R)$.  For each $i\geq 0$, define the open subgroup
\[
G^i:=\{g\in G: g\equiv I \pmod{\p^i}\}
\]
of $G$.  For each $i\geq 0$, we have a quotient group $G^{[i]}:=G^i/G^{i+1}$.   Reduction modulo $\p$ induces an injective homomorphism $\nu_0\colon G^{[0]} \hookrightarrow \GL_2(\FF)$ whose image we will denote by $\bbar{G}$.  For $i\geq 1$, we have an injective homomorphism $\nu_i\colon G^{[i]}\hookrightarrow M_2(\FF)=\gl_2(\FF)$ that takes the coset $[1+\pi^{i} B]$ to $B$ modulo $\p$; we denote its image by $\gG_i$.

Take any $i\geq 1$.  For $g=I+\pi^i B\in G^i$, we have $\det(g)\equiv 1+\pi^i \tr(B) \pmod{\p^{i+1}}$.  So for $g\in G^i$, we have  $\det(g)\equiv 1 \pmod{\p^{i+1}}$ if and only if $\nu_i([g])$ lies in $\sl_2(\FF)$.

Let $H$ be the commutator subgroup of $G$.  With notation as above, we define $\bbar{H}\subseteq \GL_2(\FF)$ and $\hH_i$ for $i\geq 1$.  Using that $H\subseteq \SL_2(R)$, we find that $\bbar{H}\subseteq \SL_2(\FF)$ and that $\hH_i\subseteq \sl_2(\FF)$ for all $i\geq 1$.

The vector spaces $\gG_i$ and $\hH_i$ are invariant under conjugation by $\bbar{G}$; this follows by considering the conjugation action of $G$ on $G^i$ and $H^i$.  The commutator map $(g,h)\mapsto ghg^{-1}h^{-1}$ induces a function $G^{[0]}\times G^{[i]} \to H^{[i]}$ that corresponds to the function
\begin{align} \label{E:comm Ggh}
\bbar{G} \times \gG_i \to \hH_i, \quad (g,x)\mapsto gxg^{-1} - x
\end{align}
via $\nu_0$ and $\nu_i$.  The commutator map also induces a function $G^{[1]}\times G^{[i]} \to H^{[i+1]}$ that corresponds to the function
\begin{align} \label{E:comm ggh}
\gG_1 \times \gG_i \to \hH_{i+1}, \quad (x,y)\mapsto [x,y]=xy-yx
\end{align}
via $\nu_0$, $\nu_i$ and $\nu_{i+1}$. 

\begin{lemma} 
With notation as above, assume that $\gG_1=\gl_2(\FF)$ and $\hH_1=\sl_2(\FF)$.  Then $H$ is the subgroup of $\SL_2(R)$ consists of those matrices whose image modulo $\lambda$ lies in $[\bbar{G},\bbar{G}]\subseteq \SL_2(\FF)$.
\end{lemma}
\begin{proof}
We will prove that $\hH_i=\sl_2(\FF)$ for all $i\geq 1$ by induction on $i$.  The base case $\hH_1=\sl_2(\FF)$ is true by assumption so suppose that $\hH_i=\sl_2(\FF)$ for some fixed $i\geq 1$.   From the map (\ref{E:comm ggh}), we find that $\hH_{i+1} \subseteq \sl_2(\FF)$ contains the $\FF$-subspace spanned by $[x,y]$ with $x\in \gG_1=\gl_2(\FF)$ and $y\in \hH_i=\sl_2(\FF)$.  We thus have $\hH_{i+1}=\sl_2(\FF)$ since $\sl_2(\FF)$ is spanned by the vectors
\[
[\left(\begin{smallmatrix} 1 & 0 \\0 & 0\end{smallmatrix}\right), \left(\begin{smallmatrix} 0 & 1 \\0 & 0\end{smallmatrix}\right) ]=\left(\begin{smallmatrix} 0 & 1 \\0 & 0\end{smallmatrix}\right),\quad [\left(\begin{smallmatrix} 1 & 0 \\0 & 0\end{smallmatrix}\right), \left(\begin{smallmatrix} 0 & 0 \\1 & 0\end{smallmatrix}\right) ]=-\left(\begin{smallmatrix} 0 & 0 \\1 & 0\end{smallmatrix}\right) \quad \text{ and }\quad
[\left(\begin{smallmatrix} 0 & 0 \\1 & 0\end{smallmatrix}\right), \left(\begin{smallmatrix} 0 & 1 \\0 & 0\end{smallmatrix}\right) ]=\left(\begin{smallmatrix} -1 & 0 \\0 & 1\end{smallmatrix}\right).
\]

Since $H$ is a closed subgroup of $\SL_2(R)$ with $\hH_i=\sl_2(\FF)$ for all $i\geq 1$, we find that $H$ contains all the $A\in \SL_2(R)$ with $A\equiv I \pmod{\lambda}$.   The lemma is now immediate since $\bbar{H}=[\bbar G,\bbar G]$.
\end{proof}

\subsection{Proof of Proposition~\ref{P:full GL2R criterion}} \label{SS:proof of full GL2R criterion}

Let $H$ be the commutator subgroup of $G$ and fix notation as in \S\ref{SS:filtration}.   

\begin{lemma}\label{L:gl1}
We have $\gG_1=\gl_2(\FF)$.
\end{lemma}
\begin{proof}
The lemma holds if $|\FF|=2$ by (\ref{P:full GL2R criterion d}), so we may assume that $|\FF|>2$.  When $|\FF|>3$, the lemma holds from Proposition~4.1 of \cite{MR2499412} .   So we may assume that $|\FF|=3$. Using (\ref{P:full GL2R criterion c}), we find that $\gG_1$ contains a nonscalar matrix.   We have $\bbar G=\GL_2(\FF)$ by (\ref{P:full GL2R criterion b}) and hence the space $\gG_1$ is invariant under conjugation by $\GL_2(\FF)$.  We thus have $\gG_1\supseteq \sl_2(\FF)$ by Lemma~\ref{L:PR adjoint}.   

We now suppose that $\gG_1\neq \gl_2(\FF)$ and hence $\gG_1=\sl_2(\FF)$ since $|\FF|$ is prime.  So for all $g\in G^1$, we have $\det(g)\equiv 1 \pmod{\p^2}$.  Let $W$ be the subgroup of $\GL_2(R)$ generated by $G^1$ and $H$; it is a normal subgroup of $G$.  Note that $\det(g)\equiv 1 \pmod{\p^2}$ for all $g\in W$ since this is true for all $g\in G^1$ and we have $H\subseteq \SL_2(R)$.  Since $\det(G)=R^\times$ by (\ref{P:full GL2R criterion a}) and $\det(W)\subseteq 1 + \p^2 R$, we find that $(R/\p^2)^\times$ is a quotient of $G/W$.

Let $\bbar{W}$ be the image of $W$ modulo $\p$.   We have 
\[
\bbar{W}=\bbar{H}=[\bbar G,\bbar G] = [\GL_2(\FF),\GL_2(\FF)]=\SL_2(\FF),
\]
where the last equality uses Lemma~\ref{L:finite field facts}(\ref{L:finite field facts iii}).  Since $W\supseteq G^1$ and $\bbar{W}=\SL_2(\FF)$, we find that the group $W$ is normal in $G$ and $G/W\cong \GL_2(\FF)/\SL_2(\FF)\cong \FF^\times$.  This is a contradiction since $(R/\p^2)^\times$ is a quotient of $G/W$ and has cardinality strictly larger than $\FF^\times=(R/\p)^\times$.  Therefore, $\gG_1=\gl_2(\FF)$.
\end{proof}

\begin{lemma} \label{L:hHi eq sl2}
We have $\hH_i=\sl_2(\FF)$ for all $i\geq 1$.
\end{lemma}
\begin{proof}
We have $\bbar G=\GL_2(\FF)$ and $\gG_1=\gl_2(\FF)$ by Lemma~\ref{L:gl1}.   By (\ref{E:comm Ggh}), $\hH_1\subseteq \sl_2(\FF)$ contains the $\FF$-subspace spanned by $gxg^{-1}-x$ with $g\in \GL_2(\FF)$ and $x\in \gl_2(\FF)$.     After computing $gxg^{-1}-x$ with  $g\in \{\left(\begin{smallmatrix} 1 & 1 \\0 & 1\end{smallmatrix}\right),\left(\begin{smallmatrix} 0 & -1 \\1 & 0\end{smallmatrix}\right)\}$ and $x\in \{\left(\begin{smallmatrix} 0 & 0 \\0 & 1\end{smallmatrix}\right),\left(\begin{smallmatrix} 0 & 0 \\1 & 0\end{smallmatrix}\right)\}$, we deduce that $\hH_1\supseteq \sl_2(\FF)$ and hence $\hH_1=\sl_2(\FF)$.

We now prove the lemma by induction on $i\geq 1$.  We have already proved the base case so suppose that $\hH_i=\sl_2(\FF)$ for some $i\geq 1$.  From the map (\ref{E:comm ggh}), we find that $\hH_{i+1} \subseteq \sl_2(\FF)$ contains the $\FF$-subspace spanned by $[x,y]$ with $x\in \gG_1=\gl_2(\FF)$ and $y\in \hH_i=\sl_2(\FF)$.  We thus have $\hH_{i+1}=\sl_2(\FF)$ since $\sl_2(\FF)$ is spanned by the vectors
\[
[\left(\begin{smallmatrix} 1 & 0 \\0 & 0\end{smallmatrix}\right), \left(\begin{smallmatrix} 0 & 1 \\0 & 0\end{smallmatrix}\right) ]=\left(\begin{smallmatrix} 0 & 1 \\0 & 0\end{smallmatrix}\right),\quad [\left(\begin{smallmatrix} 1 & 0 \\0 & 0\end{smallmatrix}\right), \left(\begin{smallmatrix} 0 & 0 \\1 & 0\end{smallmatrix}\right) ]=-\left(\begin{smallmatrix} 0 & 0 \\1 & 0\end{smallmatrix}\right) \quad \text{ and }\quad
[\left(\begin{smallmatrix} 0 & 0 \\1 & 0\end{smallmatrix}\right), \left(\begin{smallmatrix} 0 & 1 \\0 & 0\end{smallmatrix}\right) ]=\left(\begin{smallmatrix} -1 & 0 \\0 & 1\end{smallmatrix}\right).\qedhere
\]
\end{proof}

\begin{lemma} \label{L:commutator H G}
The commutator subgroup $H$ of $G$ agrees with the subgroup of matrices in $\SL_2(R)$ whose image modulo $\p$ lies in $[\GL_2(\FF),\GL_2(\FF)]$.  If $|\FF|>2$, then $H=\SL_2(R)$.   
\end{lemma}
\begin{proof}
Let $H'$ be the group of matrices in $\SL_2(R)$ whose image modulo $\p$ lies in the group $[\GL_2(\FF),\GL_2(\FF)]$.  Since $\bbar G=\GL_2(\FF)$, the image of $H'$ modulo $\p$ is equal to $[\bbar G,\bbar G]=\bbar H$.   For each $i\geq 1$, let $H'^{i}$ be the group of $g\in H'$ for which $g\equiv I \pmod{\p^i}$.   The inclusion $H\subseteq H'$ induces an injective homomorphism $H^{i}/H^{i+1} \hookrightarrow H'^{i}/H'^{i+1}$ that we view as an inclusion.

Suppose that $H\neq H'$.  The group $H'$ is open in $\SL_2(R)$ and contains $H$.  Since $H$ is a proper closed subgroup of $H'$ that has the same image modulo $\p$, we must have $H^{i}/H^{i+1} \subsetneq H'^{i+1}/H'^{i+1}$ for some $i\geq 1$.   Since $H'\subseteq \SL_2(R)$, this implies that $\hH_i\neq \sl_2(\FF)$ which contradicts Lemma~\ref{L:hHi eq sl2}.  Therefore, $H=H'$.  If $|\FF|>2$, we have $H=H'=\SL_2(R)$ by  Lemma~\ref{L:finite field facts}(\ref{L:finite field facts iii}).
\end{proof}

We claim that $G\supseteq \SL_2(R)$.   If $|\FF|>2$, then Lemma~\ref{L:commutator H G} implies that $G\supseteq H=\SL_2(R)$.  Now suppose that $|\FF|=2$.  The group $[\GL_2(\FF),\GL_2(\FF)]$ has cardinality $3$ and has index $2$ in $\SL_2(\FF)$.  By Lemma~\ref{L:commutator H G}, $H$ is the index $2$ subgroup of $\SL_2(R)$ consisting of matrices whose image modulo $\p$ lies in $[\GL_2(\FF),\GL_2(\FF)]$.   By (\ref{P:full GL2R criterion e}), there is an element $g\in G \cap \SL_2(R)$ whose image in $\SL_2(\FF)$ has order $2$.  Therefore, $g$ represents the nonidentity coset of $H$ in $\SL_2(R)$.   Since $g\in G$ and $H\subseteq G$, we have $\SL_2(R) \subseteq G$.  This completes the proof of the claim.

We thus have $G=\GL_2(R)$ since $G\supseteq \SL_2(R)$ and $\det(G)=R^\times$ by (\ref{P:full GL2R criterion a}).   The proposition follows from this and Lemma~\ref{L:commutator H G}.

\subsection{Commutator subgroups of $\GL_2(R)$ and $\SL_2(R)$}

The following gives some information on commutator subgroups that will be useful later.

\begin{prop} \label{P:commutator GL2(R)}
\begin{romanenum}
\item \label{P:commutator GL2(R) i}
If $|\FF|>2$, then the commutator subgroup of $\GL_2(R)$ is $\SL_2(R)$.  
\item  \label{P:commutator GL2(R) ii}
If $|\FF|=2$, then the commutator subgroup of $\GL_2(R)$ is
\[
\{ B \in \SL_2(R): B \text{ modulo $\p$ lies in }[\GL_2(\FF),\GL_2(\FF)]\}.
\]
In particular, $[\SL_2(R): [\GL_2(R),\GL_2(R)]]=2$.
\end{romanenum}
\end{prop}
\begin{proof}
Define $G:=\GL_2(R)$.  Note that $G$ satisfies all the conditions of Proposition~\ref{P:full GL2R criterion}.   Lemma~\ref{L:commutator H G} in the proof of Proposition~\ref{P:full GL2R criterion} shows that $[G,G]$ is the group consisting of all $B \in \SL_2(R)$ for which $B$ modulo $\p$ lies in $[\GL_2(\FF),\GL_2(\FF)]$.   This proves (\ref{P:commutator GL2(R) ii}).   Part (\ref{P:commutator GL2(R) i}) follows from Lemma~\ref{L:finite field facts}(\ref{L:finite field facts iii}).  
\end{proof}

\begin{prop} \label{P:SL2 commutator}
Suppose $|\FF|>3$,
\begin{romanenum}
\item \label{P:SL2 commutator i}
The group $\SL_2(R)$ is equal to its own commutator subgroup.
\item \label{P:SL2 commutator ii}
The only closed normal subgroup of $\SL_2(R)$ with simple quotient is the group consisting of those matrices $A \in \SL_2(R)$ for which $A\equiv \pm I \bmod{\p}$.
\end{romanenum}
\end{prop}
\begin{proof}
Define the group $G=\SL_2(R)$ and let $H$ be its commutator subgroup.  With notation as in \S\ref{SS:filtration}, we have $\bbar G=\SL_2(\FF)$ and subgroups $\hH_i\subseteq \gG_i=\sl_2(\FF)$ for all $i\geq 1$.     The image of $H$ modulo $\p$ is $\bbar H = [\bbar G, \bbar G]=[\SL_2(\FF),\SL_2(\FF)]=\SL_2(\FF)$, where the last equality uses Lemma~\ref{L:finite field facts}(\ref{L:finite field facts ii}).

Take any $i\geq 1$.   With $g=\left(\begin{smallmatrix} 0 & -1 \\1 & 0\end{smallmatrix}\right) \in \bbar G$ and $x=\left(\begin{smallmatrix} 0 & -1 \\0 & 0\end{smallmatrix}\right) \in \gG_i$, the matrix $gxg^{-1}-x=\left(\begin{smallmatrix} 0 & 1 \\1 & 0\end{smallmatrix}\right)$ lies in $\hH_i$ by (\ref{E:comm Ggh}).  In particular, $\hH_i\subseteq \sl_2(\FF)$ contains a nonscalar matrix.   The group $\hH_i$ is stable under conjugation by $\GL_2(\FF)$ since $G$ is a normal subgroup of $\GL_2(R)$.  Therefore, $\hH_i=\sl_2(\FF)$ by Lemma~\ref{L:PR adjoint}.  

We have shown that $H$ is a closed subgroup of $\SL_2(R)$ for which $\bbar H=\SL_2(\FF)$ and $\hH_i=\sl_2(\FF)$ for all $i\geq 1$.   Therefore, $H=\SL_2(R)$.  This proves (\ref{P:SL2 commutator i}).

Let $N$ be a closed normal subgroup of $\SL_2(R)$ for which $S:=\SL_2(R)/N$ is simple.  The group $S$ is finite since it is simple and profinite.   Let $\varphi\colon \SL_2(R)\to S$ be the quotient map.  The group $S$ is nonabelian since $\SL_2(R)$ is equal to its own commutator subgroup.  
Let $W$ be the closed normal subgroup of $\SL_2(R)$ consisting of all $A\in \SL_2(R)$ for which $A\equiv \pm I \pmod{\p}$.   The group $W$ is pro-solvable and hence $\varphi(W)$ is a solvable normal subgroup of $S$.  We have $\varphi(W)=1$ since $S$ is nonabelian and simple.   Therefore, $W\subseteq N$.    We have $\SL_2(R)/W \cong \SL_2(\FF)/\{\pm I\}$ which is simple by Lemma~\ref{L:finite field facts}(\ref{L:finite field facts i}).    Therefore, $W=N$ which proves (\ref{P:SL2 commutator ii}).
\end{proof}

\section{Group theoretic criterion for large adelic image}
\label{S:group theory adelic}

Let $G$ be a subgroup of $\GL_2(\widehat{A})$.  The goal of this section is to give conditions that ensure that $G$ and $\GL_2(\widehat{A})$ have the same commutator subgroup.   For a nonzero ideal $\aA$ of $A$, we will denote by $G_\aA$ the image of $G$ under the projection map $\GL_2(\widehat{A})\to \GL_2(A_\aA)$.

Let $\Lambda$ be the set of nonzero prime ideals of $A$.  

\begin{thm} \label{T:criterion for computing commutator}
Let $G$ be a closed subgroup of $\GL_2(\widehat{A})$ such that the following hold:
\begin{alphenum}
\item \label{T:criterion for computing commutator a}
For all $\lambda\in \Lambda$, we have $G_\lambda\supseteq \SL_2(A_\lambda)$.
\item \label{T:criterion for computing commutator b}
For all distinct $\lambda_1,\lambda_2\in \Lambda$ with $N(\lambda_1)=N(\lambda_2)>3$, 
$G$ modulo $\lambda_1\lambda_2$ has a subgroup of cardinality $N(\lambda_1)^2$.
\item \label{T:criterion for computing commutator c}
For all distinct $\lambda_1,\lambda_2\in \Lambda$ with $N(\lambda_1)=N(\lambda_2)=2$, the group $G_{\lambda_1\lambda_2} \cap \SL_2(A_{\lambda_1\lambda_2})$ contains a subgroup that is conjugate in $\GL_2(A_{\lambda_1\lambda_2})$ to  $\left\{ \left(\begin{smallmatrix}1 & b \\0 & 1\end{smallmatrix}\right) : b \in A_{\lambda_1\lambda_2} \right\}.$

\item \label{T:criterion for computing commutator d}
Suppose $q\in \{2,3\}$ and let $\aA$ be the ideal that is the product of the prime ideals of $A$ of norm $q$.    Then $\det(G_\aA)=A_\aA^\times$.
\end{alphenum}
Then $[G,G]=[\GL_2(\widehat{A}),\GL_2(\widehat{A})]$.  In particular, $[G,G]=\SL_2(\widehat{A})$ when $q>2$.  
\end{thm}

\subsection{Proof of Theorem~\ref{T:criterion for computing commutator}}

Define $H:=[G,G]$; it is a closed subgroup of $\SL_2(\widehat{A})$.

\begin{lemma} \label{L:independence of H at distinct places}
For any distinct nonzero prime ideals $\lambda_1$ and $\lambda_2$ of $A$ with norm at least $4$, we have $H_{\lambda_1\lambda_2}=\SL_2(A_{\lambda_1})\times \SL_2(A_{\lambda_2})$.
\end{lemma}
\begin{proof}
For each $1\leq i \leq 2$, we have $G_{\lambda_i}\supseteq \SL_2(A_{\lambda_i})$ by (\ref{T:criterion for computing commutator a}).  Since $\SL_2(A_{\lambda_i})$ is perfect by Proposition~\ref{P:SL2 commutator}(\ref{P:SL2 commutator i}), we deduce that $H_{\lambda_i}=[G_{\lambda_i},G_{\lambda_i}]$ equals $\SL_2(A_{\lambda_i})$.  We thus have an inclusion of groups $H_{\lambda_1\lambda_2} \subseteq H_{\lambda_1}\times H_{\lambda_2}=\SL_2(A_{\lambda_1})\times \SL_2(A_{\lambda_2})$ such that each projection $p_i\colon H_{\lambda_1 \lambda_2}\to \SL_2(A_{\lambda_i})$ is surjective.   Let $N_1$ and $N_2$ be the kernel of $p_2$ and $p_1$, respectively.   We may identify $N_i$ with a closed normal subgroup of $\SL_2(A_{\lambda_i})$ and hence have an inclusion $N_1\times N_2 \subseteq H_{\lambda_1\lambda_2}$.  By Goursat's lemma (\cite[Lemma 5.2.1]{MR457455}), the inclusion $H_{\lambda_1\lambda_2} \subseteq \SL_2(A_{\lambda_1})\times \SL_2(A_{\lambda_2})$ induces a homomorphism
\[
H_{\lambda_1\lambda_2}/(N_1\times N_2) \hookrightarrow \SL_2(A_{\lambda_1})/N_1 \times \SL_2(A_{\lambda_2})/N_2
\]
whose image is the graph of an isomorphism $\SL_2(A_{\lambda_1})/N_1 \xrightarrow{\sim} \SL_2(A_{\lambda_2})/N_2$.

Suppose the group $\SL_2(A_{\lambda_1})/N_1$ is trivial.  We then have $N_i=\SL_2(A_{\lambda_i})$ for each $1\leq i \leq 2$.   Therefore, $\SL_2(A_{\lambda_1})\times \SL_2(A_{\lambda_2}) =N_1\times N_2 \subseteq H_{\lambda_1\lambda_2} \subseteq \SL_2(A_{\lambda_1})\times \SL_2(A_{\lambda_2})$
and the lemma follows.    

We may now assume that $\SL_2(A_{\lambda_1})/N_1$ is nontrivial and hence each $N_i$ is a proper closed normal subgroup of $\SL_2(A_{\lambda_i})$.  By Proposition~\ref{P:SL2 commutator}(\ref{P:SL2 commutator ii}), we find that $N_i \subseteq \{B\in \SL_2(A_{\lambda_i}): B \equiv \pm I \pmod{\lambda_i}\}$.   Therefore, the homomorphism
\[
H_{\lambda_1\lambda_2} \to \SL_2(\FF_{\lambda_1})/\{\pm I\} \times\SL_2(\FF_{\lambda_2})/\{\pm I\} 
\]
obtained by composing reduction modulo $\lambda_1\lambda_2$ with the obvious quotient map has image equal to the graph of an isomorphism $\SL_2(\FF_{\lambda_1})/\{\pm I\} \xrightarrow{\sim}\SL_2(\FF_{\lambda_2})/\{\pm I\}$ of finite simple groups.  By comparing cardinalities of these simple groups, we have $N(\lambda_1)=N(\lambda_2)$.  

Now consider the homomorphism
\[
\varphi\colon G_{\lambda_1 \lambda_2} \to \GL_2(\FF_{\lambda_1})/\{\pm I\} \times \GL_2(\FF_{\lambda_2})/\{\pm I\}
\]
obtained by reducing modulo $\lambda_1\lambda_2$ and composing with the obvious quotient map.  From (\ref{T:criterion for computing commutator b}), $\varphi(G_{\lambda_1\lambda_2})$ contains a group of order $N(\lambda_1)N(\lambda_2)$; it is a $p$-Sylow subgroup of $\GL_2(\FF_{\lambda_1})/\{\pm I\} \times \GL_2(\FF_{\lambda_2})/\{\pm I\}$ where $p$ is the prime dividing $q$.  In particular, there is a $g\in G_{\lambda_1\lambda_2}$ such that $\varphi(g)=(I,g_2)$, where $g_2 \in \GL_2(\FF_{\lambda_2})/\{\pm I\}$ has order a positive power of $p$.
We have already shown that $\varphi(H_{\lambda_1\lambda_2})$ is the graph of an isomorphism $\SL_2(\FF_{\lambda_1})/\{\pm I\} \xrightarrow{\sim}\SL_2(\FF_{\lambda_2})/\{\pm I\}$.  So there is an $(h_1,h_2) \in \varphi(H_{\lambda_1\lambda_2})$ for which $h_2$ and $g_2$ do not commute (an element in $\GL_2(\FF_{\lambda_2})/\{\pm I\}$ that commutes with $\SL_2(\FF_{\lambda_2})/\{\pm I\}$ will be represented by a scalar matrix and hence has order relatively prime to $p$).  Since $H$ is normal in $G$, we deduce that 
\[
(I,g_2) (h_1,h_2) (I, g_2)^{-1} = (h_1,g_2h_2g_2^{-1})
\]
 is also in $\varphi(H_{\lambda_1\lambda_2})$.  Since $(h_1,g_2h_2g_2^{-1})$ and $(h_1,h_2)$ are distinct elements of $\varphi(H_{\lambda_1\lambda_2})$, this contradicts that the group $\varphi(H_{\lambda_1\lambda_2})$ is the graph of a function.
\end{proof}

\begin{lemma} \label{L:Goursat with perfect groups}
Let $S_1,\ldots, S_r$ be profinite groups that are all perfect with $r>1$.    Let $H$ be a closed subgroup of $S_1\times \cdots \times S_r$ such that the projection $H\to S_i\times S_j$ is surjective for all $1\leq i < j \leq r$.  Then $H=S_1\times \cdots \times S_r$.
\end{lemma}
\begin{proof}
When the $S_i$ are finite, this is \cite[Lemma 5.2.2]{MR457455} and follows from Goursat's lemma.  The general case follows directly from the finite group case since $H$ is closed.
\end{proof}

\begin{lemma} \label{L:surjectivity of SL2 via H}
Let $\Lambda_1$ be the set of nonzero prime ideals of $A$ of norm at least $4$.  Then the projection $H \to \prod_{\lambda\in \Lambda_1} \SL_2(A_\lambda)$ is surjective.
\end{lemma}
\begin{proof}
Let $I$ be any finite nonempty subset of $\Lambda_1$ with cardinality at least $2$.  The group $\SL_2(A_\lambda)$ is perfect for all $\lambda\in \Lambda_1$ by Proposition~\ref{P:SL2 commutator}(\ref{P:SL2 commutator i}).   For any two distinct $\lambda_1,\lambda_2\in I$, the projection $H \to \SL_2(A_{\lambda_1}) \times \SL_2(A_{\lambda_2})$ is surjective by Lemma~\ref{L:independence of H at distinct places}.  Therefore, the projection $H\to \prod_{\lambda\in I} \SL_2(A_{\lambda})$ is surjective by Lemma~\ref{L:Goursat with perfect groups}.   The lemma follows by increasing the set $I$ and using that $H$ is a closed subgroup of $\SL_2(\widehat{A})=\prod_\lambda \SL_2(A_\lambda)$.
\end{proof}

\begin{lemma} \label{L:surjectivity of SL2 via H 2 and 3}
Let $\Lambda_2$ be the set of nonzero prime ideals of $A$ of norm at most $3$.  Then the projection $H \to \prod_{\lambda\in \Lambda_2} [\GL_2(A_\lambda),\GL_2(A_\lambda)]$ is surjective.
\end{lemma}
\begin{proof}
We may assume that  $\Lambda_2$ is nonempty and hence $q \in \{2,3\}$.    We claim that the projection
\[
G\to \prod_{\lambda\in \Lambda_2} \GL_2(A_\lambda)
\]
is surjective.   The lemma follow immediately by taking commutator subgroups.

Suppose on the contrary that the claim fails. 
 Then there is a minimal nonempty set $\Lambda_2'\subseteq \Lambda_2$ for which $G \to \prod_{\lambda\in \Lambda_2'} \GL_2(A_\lambda)$ is not surjective and we denote its image by $B$.   For any $\lambda\in \Lambda_2'$, we have $G_{\lambda}=\GL_2(A_\lambda)$ by (\ref{T:criterion for computing commutator a}) and (\ref{T:criterion for computing commutator d}), and hence $|\Lambda_2'|\geq 2$.  
 
 Fix a place $\lambda_1\in \Lambda_2'$ and define the groups $B_1:=\GL_2(A_{\lambda_1})$ and $B_2:=\prod_{\lambda\in \Lambda_2'-\{\lambda_1\}}\GL_2(A_\lambda)$.  We can view $B$ as a subgroup of $B_1 \times B_2$.  The projections $p_i\colon B\to B_i$ are surjective by the minimality of $\Lambda_2'$.  Let $N_1$ and $N_2$ be the kernels of $p_2$ and $p_1$, respectively.  We can view $N_i$ as a subgroup of $B_i$ and hence $N_1\times N_2 \subseteq B$.  The image of the quotient map $G\to B_1/N_1 \times B_2/N_2$ is the graph of an isomorphism $B_1/N_1 \xrightarrow{\sim} B_2/N_2$ by Goursat's lemma (\cite[Lemma 5.2.1]{MR457455}).
 
 If $B_1/N_1$ or $B_2/N_2$ is trivial, then $N_1=B_1$ and $N_2=B_2$ and hence $B\supseteq B_1 \times B_2= \prod_{\lambda\in \Lambda_2'} \GL_2(A_\lambda)$ which contradicts our choice of $\Lambda_2'$.   
 
So we may assume each $N_i$ is a proper closed normal subgroup of $B_i$.  There are thus proper closed normal subgroups $M_i$ of $B_i$ with $M_i\supseteq N_i$ such that the image of $G\to B_1/M_1 \times B_2/M_2$ is the graph of an isomorphism $B_1/M_1 \xrightarrow{\sim} B_2/M_2$ of finite simple groups.  The group $B_1$ is prosolvable (this uses that $\GL_2(\FF_2)$ and $\GL_2(\FF_3)$ are solvable).  Therefore, $B_1/M_1$ is a cyclic group of prime order.

Suppose that $q=3$.  Using that each $B_i/M_i$ is abelian and Proposition~\ref{P:SL2 commutator}(\ref{P:SL2 commutator i}), we find that $M_1 \supseteq \SL_2(A_{\lambda_1})$ and $M_2 \supseteq \prod_{\lambda\in \Lambda_2'-\{\lambda_1\}} \SL_2(A_\lambda)$.  Since the homomorphism $G\to B_1/M_1 \times B_2/M_2$ is not surjective, we deduce that the projection $\det(G) \to \prod_{\lambda\in \Lambda_2'} A_\lambda^\times$ is not surjective.  This contradicts (\ref{T:criterion for computing commutator d}).

Finally suppose that $q=2$.  We have $\Lambda_2=\Lambda_2'=\{\lambda_1,\lambda_2\}$ for a unique $\lambda_2$.  Since each $B_i/M_i$ is abelian and $B_i = \GL_2(A_{\lambda_i})$, we have $M_i \supseteq [\GL_2(A_{\lambda_i}),\GL_2(A_{\lambda_i})]$.   By (\ref{T:criterion for computing commutator d}), there is a $g\in G_{\lambda_1\lambda_2} \cap \SL_2(A_{\lambda_1\lambda_2})$ whose projection $g_1$ in $\GL_2(A_{\lambda_1})$ has order $2$ modulo $\lambda_1$ and whose projection in $\GL_2(A_{\lambda_2})$ is the identity matrix.  We have $g_1 \in N_1 \subseteq M_1$.  Using Proposition~\ref{P:commutator GL2(R)}(\ref{P:commutator GL2(R) ii}), the group $\SL_2(A_{\lambda_1})$ is generated by $g_1$ and $[\GL_2(A_{\lambda_1}),\GL_2(A_{\lambda_1})]$.  Therefore, $M_1\supseteq \SL_2(A_{\lambda_1})$.  A similar argument shows that $M_2\supseteq \SL_2(A_{\lambda_2})$.  Since the homomorphism $G\to B_1/M_1 \times B_2/M_2$ is not surjective, we deduce that the projection $\det(G) \to \prod_{\lambda\in \Lambda_2'} A_\lambda^\times$ is not surjective.  This contradicts (\ref{T:criterion for computing commutator d}).
\end{proof}

Define $B_1:=\prod_{\lambda \in \Lambda_1} \SL_2(A_{\lambda})$ and $B_2:=\prod_{\lambda\in \Lambda_2} [\GL_2(A_\lambda),\GL_2(A_\lambda)]$.   We have a natural inclusion
\[
H \subseteq [\GL_2(\widehat{A}),\GL_2(\widehat{A})]= \prod_{\lambda\in \Lambda_1 \cup \Lambda_2}  [\GL_2(A_\lambda),\GL_2(A_\lambda)] = B_1 \times B_2,
\]
where we have used Proposition~\ref{P:SL2 commutator}(\ref{P:SL2 commutator i}).   The projections $H\to B_1$ and $H\to B_2$ are surjective by Lemmas~\ref{L:surjectivity of SL2 via H} and \ref{L:surjectivity of SL2 via H 2 and 3}.   

Suppose $H$ is a proper subgroup of $B_1\times B_2$.  By Goursat's lemma (\cite[Lemma 5.2.1]{MR457455}), there are closed proper normal subgroups $N_i$ of $B_i$ for which we have an isomorphism $B_1/N_1\cong B_2/N_2$.   This implies that there is a finite simple group $Q$ that shows up as a quotient of both $B_1$ and $B_2$.  The group $B_1$ is perfect by Proposition~\ref{P:SL2 commutator}(\ref{P:SL2 commutator i}) so $Q$ is nonabelian.   However, the group $B_2$ is prosolvable (since $GL_2(\FF_2)$ and $\GL_2(\FF_3)$ are solvable) and hence $Q$ is cyclic.  This gives a contradiction and thus $H=B_1\times B_2=[\GL_2(\widehat{A}),\GL_2(\widehat{A})]$.  Finally when $q>2$, we have $[\GL_2(\widehat{A}),\GL_2(\widehat{A})]=\SL_2(\widehat{A})$ by Theorem~\ref{P:commutator GL2(R)}(\ref{P:commutator GL2(R) i}).

\section{Local fields and the image of inertia} \label{S:inertia}

Fix a nonzero prime ideal $\p$ of $A$.  Let $K$ be a finite separable extension of $F_\p$ which we consider as an $A$-field via the inclusions $A\subseteq F_\p\subseteq K$.   The integral closure of $A_\p$ in $K$ is a complete discrete valuation ring $\OO$ whose maximal ideal we will denote by $\m$.  Define the residue field $\FF:=\OO/\m$.

Let $v \colon K^\times \to \ZZ$ be the discrete valuation corresponding to $\OO$ normalized so that $v(K^\times)=\ZZ$ and we set $v(0)=+\infty$.    We will also denoted by $v$ the corresponding $\QQ$-valued extension of $v$ to a fixed separable closure $K^\sep$ of $K$.   Let $I_K$ be the inertia subgroup of $\Gal_K=\Gal(K^\sep/K)$ and let $K^\un$ be the maximal unramified extension of $K$ in $K^\sep$.  

Let $\phi\colon A \to K\{\tau\}$ be a Drinfeld $A$-module of rank $r$.  We  shall say that $\phi$ is \defi{defined over $\OO$} if $\phi_a\in \OO\{\tau\}$ for all $a\in A$.   The Drinfeld module $\phi$ has \defi{stable reduction} (of rank $r'$) if there exists a Drinfeld module $\phi'\colon A \to K\{\tau\}$ defined over $\OO$ such that $\phi'$ and $\phi$ are isomorphic over $K$ and the reduction of $\phi'$ modulo $\m$ is a Drinfeld module $A\to \FF\{\tau\}$ of rank $r'\geq 1$.  Recall that $\phi$ has \defi{good reduction} if it has stable reduction of rank $r$.

Suppose that $\phi$ has rank $2$.  The \defi{$j$-invariant} of $\phi$ is $j_\phi := a_1^{q+1}/a_2 \in K$, where $\phi_t = t + a_1 \tau + a_2 \tau^2$.  The Drinfeld module $\phi$ has {potentially good reduction} if and only if $v(j_\phi)\geq 0$, cf.~\cite{MR2020270}*{Lemma~5.2}.

\subsection{Image of inertia}

Let $\phi\colon A \to K\{\tau\}$ be a Drinfeld $A$-module of rank $2$.    For a nonzero ideal $\aA\subseteq A$, the Galois action on the $\aA$-torsion of $\phi$ gives rise to a Galois representation $\bbar\rho_{\phi,\aA}\colon \Gal_K \to \GL_2(A/\aA)$.  If $\phi$ has good reduction and $\p\nmid \aA$, then $\bbar\rho_{\phi,\aA}(I_K)=1$.   We now study the group $\bbar\rho_{\phi,\p}(I_K)$ when $\phi$ has good reduction.

\begin{prop} \label{P:good inertia}
Assume that $K/F_\p$ is unramified.
Let $\phi\colon A\to K\{\tau\}$ be a Drinfeld module of rank $2$ that has good reduction.  Then one of the following hold:
\begin{alphenum}
\item \label{P:good inertia a}
$\bbar\rho_{\phi,\p}(I_K)$ is conjugate in $\GL_2(\FF_\p)$ to a subgroup of $\left\{ \left(\begin{smallmatrix} a & b \\0 & 1\end{smallmatrix}\right) : a\in \FF_\p^\times, b\in \FF_\p \right\}$,
\item \label{P:good inertia b}
$\bbar\rho_{\phi,\p}(I_K)$ is a cyclic subgroup of $\GL_2(\FF_\p)$ of order $q^{2\deg \p}-1$. 
\end{alphenum}
\end{prop}
\begin{proof}
We summarize material from \S2 of \cite{MR2499411}.  After replacing $\phi$ by an isomorphic Drinfeld module, we may assume that $\phi$ is defined over $\OO$ and that reducing modulo $\m$ gives a Drinfeld module of rank $2$.    Then $\phi[\p]$ extends to a finite flat group scheme over $\OO$.  The connected and \'etale components of $\phi[\p]$ give an exact sequence $0 \to \phi[\p]^0 \to \phi[\p]\to \phi[\p]^{\text{\'et}} \to 0$ of finite flat group schemes.  Taking $K^\sep$-points gives a short exact sequence 
\begin{align} \label{E:connnected-etale}
0 \to \phi[\p]^0(K^\sep) \to \phi[\p](K^\sep)\to \phi[\p]^{\text{\'et}}(K^\sep) \to 0
\end{align}
of $\FF_\p$-vector spaces that is $\Gal_K$-equivariant.   Let $h$ be the height of the Drinfeld module $\phi$ modulo $\p$.  The $\FF_\p$-vector space $\phi[\p]^\circ(K^\sep)$ has dimension $h$.

The action of $I_K$ on $\phi[\p]^{\text{\'et}}(K^\sep)$ is trivial from the definition of an \'etale group scheme.  So when $h=1$, (\ref{P:good inertia a}) follows from the exact sequence (\ref{E:connnected-etale}).    We may now suppose that $h=2$ and hence $\phi[\p]^{\text{\'et}}(K^\sep)=0$.  Property (\ref{P:good inertia b}) then follows from \cite[Proposition~2.7(ii)]{MR2499411} which shows that $I_K$ acts on $\phi[\p]^0(K^\sep)$ via a fundamental character whose image is cyclic of order $q^{2\deg \p} -1$. 
\end{proof}

The following proposition, which we prove in \S\ref{SS:main Tate new proof}, gives constraints on $\bbar\rho_{\phi,\aA}(I_K)$ when $\phi$ has stable and bad reduction.  

\begin{prop} \label{P:main Tate new} 
Let $\phi \colon A \to K\{\tau\}$ be a Drinfeld module of rank $2$ that has stable reduction of rank $1$.  Consider an ideal $\aA=\p^e \aA'$, where $e\geq 0$ is an integer and $\aA'$ is a nonzero ideal of $A$ that is relatively prime to $\p$. 
\begin{romanenum}
\item \label{P:main Tate new i} 
 The group $\bbar\rho_{\phi,\aA}(\Gal_K)$ is conjugate in $\GL_2(A/\aA)$ to a subgroup of 
\begin{align} \label{E:Tate Borel}
\left\{ \left(\begin{smallmatrix} a & b \\0 & c\end{smallmatrix}\right) : a\in (A/\aA)^\times \text{ with } a\equiv 1 \pmod{\aA'},\,  b\in A/\aA,\, c\in \FF_q^\times\right\}.
\end{align}
\item \label{P:main Tate new ii} 
The cardinality of $\bbar\rho_{\phi,\aA}(I_K)$ is divisible by the denominator of $\tfrac{v(j_\phi)}{N(\aA)}\in \QQ$ in lowest terms.
\item \label{P:main Tate new iii} 
If $\gcd(v(j_\phi),q)=1$ and $e\leq 1$, then $\bbar\rho_{\phi,\aA}(I_K)$ contains a subgroup that is conjugate in $\GL_2(A/\aA)$ to $\left\{ \left(\begin{smallmatrix} 1 & b \\0 & 1\end{smallmatrix}\right) : b\in A/\aA\right\}$.
\end{romanenum}
\end{prop}

\subsection{Proof of Proposition~\ref{P:main Tate new}} \label{SS:main Tate new proof}
After possibly replacing $\phi$ with a $K$-isomorphic Drinfeld module, we may assume that $\phi$ is defined over $\OO$ and that the reduction of $\phi$ modulo $\m$ is a Drinfeld module of rank $1$.    We have $j_\phi\neq 0$ since $\phi$ has stable reduction of rank $1$.

For a Drinfeld module $\psi\colon A\to K\{\tau\}$, a \defi{$\psi$-lattice} is a finitely generated projective $A$-submodule $\Gamma$ of ${}^\psi\!K^\sep$ that is discrete and is stable under the action of $\Gal_K$.    By discrete we mean that any disk of finite radius in $K^\sep$, with respect to the valuation $v$, contains only finitely many elements of $\Gamma$.  

Associated to $\phi$, we now recall the construction of its \emph{Tate uniformization}; it is a pair $(\psi,\Gamma)$ which consists of a Drinfeld module $\psi\colon A \to K\{\tau\}$ of rank $1$ defined over $\OO$ and a $\psi$-lattice $\Gamma$ of rank $1$.  For details see Proposition 7.2 of \cite{MR0384707} and its proof; for further details see \cite[Chapter~4 \S3]{MR2488548}.  There exists a unique Drinfeld module $\psi\colon  A\to K\{\tau\}$ defined over $\OO$ of rank $1$ and a unique series
$u= \tau^0 + \sum_{i=1}^\infty a_i \tau^i \in \OO\{\!\{\tau\}\!\}$
with $a_i\in \m$ and $a_i \to 0$ in $K$, such that 
\begin{equation} \label{E:Tate uniformization 1}
u \psi_a=\phi_au
\end{equation}
for all $a\in A$.  We can identify $u$ with the power series $u(x)=x+\sum_{i=1}^\infty a_i x^{q_i}$ and one can show that $u(z)$ converges for all $z\in K^\sep$.  By considering the analytic properties of $u$ and (\ref{E:Tate uniformization 1}), one shows that the map
\begin{align} \label{E:u map}
{}^\psi\!K^\sep \to  {}^\phi\!K^\sep, \quad z\mapsto u(z)
\end{align}
is a surjective homomorphism of $A$-modules whose kernel $\Gamma$ is a $\psi$-lattice.   Since $\phi$ has rank $2$ and has stable reduction of rank $1$, the $A$-module $\Gamma$ has rank $2-1=1$.   

Fix an $a \in A$ for which $\aA=(a)$.  From the homomorphism (\ref{E:u map}) of $A$-modules, we obtain an isomorphism 
\begin{equation} \label{E:Tate uniformization torsion}
\psi_a^{-1}(\Gamma)/\Gamma \xrightarrow{\sim} \phi[a]=\phi[\aA], \quad z + \Gamma \mapsto u(z)
\end{equation}
of $A$-modules that is $\Gal_K$-equivariant.  We also have a $\Gal_K$-equivariant short exact sequence of $A$-modules:  
\begin{equation} \label{E:SES Tate}
0 \to \psi[a]=\psi_a^{-1}(0) \to \psi_a^{-1}(\Gamma)/\Gamma \xrightarrow{\psi_a} \Gamma/a\Gamma \to 0.
\end{equation}
So by combining (\ref{E:Tate uniformization torsion}) and (\ref{E:SES Tate}), we obtain a short exact sequence 
\begin{align} \label{E:nice ses}
0 \to \psi[\aA] \to \phi[\aA] \to \Gamma/\aA \Gamma \to 0
\end{align}
of $A$-modules that is $\Gal_K$-equivariant.  

The $A/\aA$-module $\psi[\aA]$ is free of rank $1$ since $\psi$ has rank $1$.  Define the character $\chi_1:=\bbar\rho_{\psi,\aA}\colon \Gal_K \to \Aut_A(\psi[\aA])=(A/\aA)^\times$.  Since $A$ is a PID, $\Gamma$ is a free $A$-module of rank $1$.  The Galois action on $\Gamma$ is thus given by a character $\chi_2\colon \Gal_K \to \Aut_A(\Gamma)=A^\times=\FF_q^\times$.   The character $\chi_2$ also describes the Galois action on the quotient $\Gamma/\aA\Gamma$.   From (\ref{E:nice ses}) we may assume, after making an appropriate choice of basis of $\phi[\aA]$, that
\[
\bbar\rho_{\phi,\aA}(\sigma) = \left(\begin{smallmatrix} \chi_1(\sigma) & * \\0 & \chi_2(\sigma)\end{smallmatrix}\right)
\]
holds for all $\sigma\in \Gal_K$.

To complete the proof of (\ref{P:main Tate new i}), it remains to show that $\chi_1(\sigma)\equiv 1 \pmod{\aA'}$ for all $\sigma\in I_K$.  Equivalently, we need to show that the action of $I_K$ on $\psi[\aA']$ is trivial.  It thus suffices to prove that $\psi$ has good reduction.  We have $\psi_t \equiv \phi_t \pmod{\m}$ by reducing (\ref{E:Tate uniformization 1}) modulo $\m$.   This proves that $\psi$ has good reduction since $\phi$ modulo $\m$ is a Drinfeld module of rank $1$ and $\psi$ has rank $1$.

We now prove (\ref{P:main Tate new ii}).   Fix a generator $\gamma$ of the $A$-module $\Gamma$ and choose a $z\in K^\sep$ for which $\psi_a(z)=\gamma$.   

\begin{lemma} \label{L:vz}
We have $v(z)=\tfrac{v(j_\phi)}{(q-1)N(\aA)}$.
\end{lemma}
\begin{proof}
First suppose that $v(z)\geq 0$.  Since $\psi_a$ has coefficients in $\OO$, we have $v(\gamma)=v(\psi_a(z))\geq 0$.  Since $\gamma$ is nonzero and $0=u(\gamma)=\gamma + \sum_{i=1}^\infty a_i \gamma^{q^i}$, we must have $v(\gamma)\geq v(a_i \gamma^{q^i})$ for some $i\geq 1$ with $a_i\neq 0$. So $v(\gamma)\geq v(a_i)+q^i v(\gamma)> q^i v(\gamma)$ which is impossible since $v(\gamma)\geq 0$.   Therefore, $v(z)<0$ and hence
\[
v(\gamma) = v(\psi_a(z))=v(z^{q^{\deg a}}) = q^{\deg a} v(z) = N(\aA) v(z).
\]
By \cite{MR2020270}*{Lemma~5.3}, we have $v(\gamma) = v(j_\phi)/(q-1)$ and the lemma follows.
\end{proof}

Let $d$ be the denominator of $v(j_\phi)/N(\aA)\in \QQ$.   Let $K^t$ be the maximal tamely ramified extension of $K$ in $K^\sep$.    Let $L$ be the minimal extension of $K^t$ in $K^\sep$ for which $\Gal(K^\sep/L)$ fixes $z+\Gamma$.  The Galois group $\Gal(K^\sep/K^t)$ acts trivially on $\Gamma$ since the Galois action on $\Gamma$ is given by $\chi_2$.  Therefore, $L=K^t(z)$.      From the isomorphism (\ref{E:Tate uniformization torsion}), we find that $K^t(z)\subseteq K^t(\phi[\aA])$.  

A rational number occurs as $v(\alpha)$ for some nonzero $\alpha\in K^t$ if and only if its denominator is relatively prime to $q$.  By Lemma~\ref{L:vz}, we find that $d$ is the minimal positive integer for which $d v(z) \in v(K^t-\{0\})$.  Therefore, $[K^t(z):K^t]$ is divisible by $d$.   So $d$ divides $[K^t(\phi[\aA]):K^t]$ and hence also divides $|\bbar\rho_{\phi,\aA}(I_K)|$.  This completes the proof of (\ref{P:main Tate new ii}).

It remains to prove (\ref{P:main Tate new iii}), so assume $\gcd(v(j_\infty),q)=1$ and $e\leq 1$.  Let $G\subseteq \GL_2(A/\aA)$ be the subgroup (\ref{E:Tate Borel}).   Since $e \leq 1$, the group $B:=\left\{ \left(\begin{smallmatrix} 1 & b \\0 & 1\end{smallmatrix}\right) : b\in A/\aA\right\}$ is a $p$-Sylow subgroup of $G$, where $p$ is the prime dividing $q$.   We have $\bbar\rho_{\phi,\aA}(I_\p) \subseteq G$ from (\ref{P:main Tate new i}), so it suffices to show that $\bbar\rho_{\phi,\aA}(I_\p)$ contains a subgroup of order $N(\aA)$ (since this will be a $p$-Sylow subgroup of $G$ and hence conjugate to $B$).   The group $\bbar\rho_{\phi,\aA}(I_\p)$ contains a subgroup of order $N(\aA)$ by (\ref{P:main Tate new i}) and our assumption  $\gcd(v(j_\infty),q)=1$.

\section{Frobenius polynomials}
\label{S:Frobenius polynomials}

Consider a Drinfeld $A$-module $\phi\colon A \to F\{\tau\}$ of rank $r$.    Take any nonzero prime ideal $\p$ of $A$ for which $\phi$ has good reduction.  So after replacing $\phi$ by an isomorphic Drinfeld, we may assume that the coefficients of $\phi$ are integral at $\p$ and that reducing modulo $\p$ gives a Drinfeld module $\bbar\phi\colon A \to \FF_\p\{\tau\}$ of rank $r$.   

Let $P_{\phi,\p}(x)\in A[x]$ be the characteristic polynomial of the Frobenius endomorphism $\pi_\p :=\tau^{\deg \p} \in \End_{\FF_\p}(\bbar\phi)$; it is the degree $r$ polynomial that is a power of the minimal polynomial of $\pi_\p$ over $F$.  The following is an easy consequence of \cite[Theorem~4.12.12]{MR1423131}.

\begin{prop}
Let $\p$ be a nonzero prime ideal of $A$ for which $\phi$ has good reduction.  For any nonzero ideal $\aA$ of $A$ that is relatively prime to $\p$, $\bbar\rho_{\phi,\aA}$ is unramified at $\p$ and 
\[
P_{\phi,\p}(x)\equiv \det(xI - \bbar\rho_{\phi,\aA}(\Frob_\p)) \pmod{\aA}.
\]
\end{prop}

We will need to explicitly know some Frobenius polynomials in order to prove Theorem~\ref{T:example}.
\begin{lemma} \label{L:traces for degree 1 primes new}
With $q$ fixed, let $\phi\colon A\to F\{\tau\}$ be the rank $2$ Drinfeld module from Theorem~\ref{T:example}.
\begin{romanenum}
\item \label{L:traces for degree 1 primes new i}
If $q\neq 2$, then
\[
P_{\phi,(t-c)}(x)= x^2-x+(t-c)
\]
for all nonzero $c\in \FF_q$.
\item \label{L:traces for degree 1 primes new ii}
If $q=3$, then $P_{\phi,(t^2+t+2)}(x)=x^2 + 2x + t^2 + t + 2$.
\item \label{L:traces for degree 1 primes new iii}
If $q=2$, then $P_{\phi,(t)}(x)=x^2 + t$ and $P_{\phi,(t+1)}(x)=x^2 + x + t + 1$.
\end{romanenum}
\end{lemma}
\begin{proof}
We first assume that $q\neq 2$ and prove (\ref{L:traces for degree 1 primes new i}).   Define the prime ideal $\q:=(t-c)$ of $A$; note that $\phi$ has good reduction at $\q$ since $c$ is nonzero.    The image of $t$ in $\FF_\q$ is $c$, so the reduction of $\phi$ modulo $\q$ is the Drinfeld $A$-module $\bbar\phi\colon A\to \FF_\q\{\tau\}=\FF_q\{\tau\}$ for which $t\mapsto c + \tau - c^{q-1} \tau^2=c+\tau - \tau^2$.   The $j$-invariant $j_{\bbar\phi}$ of $\bbar\phi$ is $-1$. By \cite[Proposition~2.11]{MR2366959}, the constant term of $P_{\phi,(t-c)}(x)$ is equal to $- (-1)^{-1} (t-c)=t-c$.  We have $P_{\phi,(t-c)}(x)=x^2-a x + (t-c)$ for a unique $a\in A$.  By \cite[Proposition~2.14(ii)]{MR2366959} (with $n=1$), we have $a\in \FF_q$ and $a=- 1^{-1} j_{\bbar{\phi}}=1$.  Therefore, $P_{\phi,(t-c)}(x)=x^2- x + (t-c)$.

Parts (\ref{L:traces for degree 1 primes new ii}) and (\ref{L:traces for degree 1 primes new iii}) both were found using the algorithm outlined in \cite[\S3.4]{MR2366959} to compute $P_{\phi,\p}$.   
\end{proof}

\section{Determinant of the image of Galois}
\label{S:determinant}

Fix polynomials $a_1$ and $a_2$ in $A$ with $a_2\neq 0$, and let $d$ be the degree of $a_2$. Let $\phi\colon A\to F\{\tau\}$ be the Drinfeld $A$-module of rank $2$ for which $\phi_t=t+a_1 \tau + a_2\tau^2$.   The following result gives an explicit expression for the index of $\det(\rho_\phi(\Gal_F))$ in $\widehat{A}^\times$.  

\begin{thm} \label{T:det image}
  Let $\zeta\in \FF_q^\times$ be the leading coefficient of $(-1)^{d+1} a_2$ in $A=\FF_q[t]$ and let $e$ be the order of $\zeta$ in $\FF_q^\times$.    Then
\[
[\widehat{A}^\times : \det(\rho_\phi(\Gal_F))] = \gcd(d-1,(q-1)/e).
\]
\end{thm}

We also give a condition that can be used to show that the image of $\bbar\rho_{\phi,\aA}$ has maximal determinant.

\begin{prop} \label{P:det image}
Let $\aA$ be a nonzero ideal of $A$ and define 
\[
g:=\gcd\big(\{d-1,q-1\} \cup \{ v_\p(a_2) : \p\nmid \aA \text{ nonzero prime ideal of $A$}\}\big).
\]
If $g=1$, then $\det(\bbar\rho_{\phi,\aA}(\Gal_F))=(A/\aA)^\times$.  
\end{prop}

We will prove Theorem~\ref{T:det image} and Proposition~\ref{P:det image} in \S\ref{SS:det proofs}.

\subsection{Galois image for rank $1$ Drinfeld modules}
Fix a nonzero $\Delta\in A$ and let $\psi\colon A\to F\{\tau\}$ be the Drinfeld $A$-module of rank $1$ for which $\psi_t=t+\Delta\tau$.   For each nonzero ideal $\aA$ of $A$, $\psi[\aA]$ is a free $A/\aA$-module of rank $1$ with a Galois action that is described by a Galois representation
\[
\bbar\rho_{\psi,\aA}\colon \Gal_F \to \Aut_{A}(\psi[\aA])=(A/\aA)^\times.
\]
Taking the inverse limit, we obtain an isomorphism $\rho_{\psi}\colon \Gal_F \to \widehat{A}^\times$.  We now state a theorem of Gekeler that describes the index of the images of $\bbar\rho_{\psi,\aA}$ and $\rho_{\psi}$.

\begin{thm} \label{T:Gekeler}
Let $d$ be the degree of $\Delta$.  Let $\zeta \in \FF_q^\times$ be the leading coefficient of $(-1)^d\Delta$ and let $e$ be its order in $\FF_q^\times$. 
\begin{romanenum}
\item \label{T:Gekeler i}
Take any nonzero ideal $\aA$ of $A$.   The integer $[(A/\aA)^\times: \bbar\rho_{\psi,\aA}(\Gal_F)]$ is the greatest common divisor of $d-1$, $(q-1)/e$, and the integers $v_\p(\Delta)$ for nonzero prime ideals $\p\nmid \aA$ of $A$.
\item \label{T:Gekeler ii}
The group $\rho_{\psi}(\Gal_F)$ has finite index in $\widehat{A}^\times$ and 
\[
[\widehat{A}^\times: \rho_{\psi}(\Gal_F)] = \gcd(d-1,(q-1)/e).
\]
\end{romanenum}
\end{thm}

\begin{proof}
Before citing Gekeler's result (Theorem~3.13 of \cite{MR3459573}), we need to match his assumptions and notation.  The representation $\bbar\rho_{\psi,\aA}$ depends only on $\aA$ and the class of $\Delta$ in $F^\times/(F^\times)^{q-1}$; this can be deduced for example from \cite[Lemma 1.8]{MR3459573}.   So after dividing $\Delta$ by a suitable $(q-1)$-th power of a monic polynomial, we may assume that 
 \[
 \Delta=(-1)^{d} \zeta P_1^{k_1}\cdots P_s^{k_s}
 \]
 where the $P_i \in A$ are distinct monic irreducible polynomials with $1\leq k_i <q-1$.   Note that this does not change the gcds in (\ref{T:Gekeler i}) and (\ref{T:Gekeler ii}).
 
 We may order the irreducible polynomials so that $P_1,\ldots, P_r$ divide $\aA$ and $P_{r+1},\ldots, P_s$ are relatively prime to $\aA$.   Choose a generator $c$ of the cyclic group $\FF_q^\times$.  We have $(-1)^d \zeta = c^{k_0}$ and $\zeta=c^{k_0^*}$ for unique $0\leq k_0, k_0^* <q-1$.  We have $k_0^*=k_0$ when $(-1)^d=1$, i.e., when $q$ is even or $d$ is even.  When $q$ is odd and $d$ is odd, we have $k_0^*\equiv (q-1)/2+k_0 \pmod{q-1}$.   Theorem 3.13 of \cite{MR3459573} now says that 
\[
[(A/\aA)^\times: \bbar\rho_{\psi,\aA}(\Gal_F)] = \gcd(d-1,q-1,k_0^*,k_{r+1},\ldots, k_s).
\]
Since $\zeta=c^{k_0^*}$ and $c$ has order $q-1$ in $\FF_q^\times$, we have $e=(q-1)/\gcd(q-1,k_0^*)$.  So $\gcd(q-1,k_0^*)=(q-1)/e$ and hence 
\[
[(A/\aA)^\times: \bbar\rho_{\psi,\aA}(\Gal_F)] = \gcd(d-1,(q-1)/e,k_{r+1},\ldots, k_s).
\]
Part (\ref{T:Gekeler i}) now follows since the set $\{k_{r+1},\ldots, k_s\}$ consists of those nonzero integers of the form $v_\p(\Delta)$ for some nonzero prime ideal $\p\nmid \aA$ of $A$.

When $\aA$ is divisible by all the irreducible polynomials $P_1,\ldots, P_s$, we have simply $[(A/\aA)^\times: \bbar\rho_{\psi,\aA}(\Gal_F)] = \gcd(d-1,(q-1)/e)$.  Part (\ref{T:Gekeler ii}) is now immediate since $\rho_\psi(\Gal_F)$ is a closed subgroup of $\widehat{A}^\times$.
\end{proof}

\subsection{Proofs of Theorem~\ref{T:det image} and Proposition~\ref{P:det image}} \label{SS:det proofs}

Let $\psi\colon A\to F\{\tau\}$ be the rank $1$ Drinfeld module for which $\psi_t=t-a_2\tau$.  By Corollary 4.6 in \cite{MR1223025}, we have
\begin{align} \label{E:Weil pairing consequence}
\det \rho_{\phi} = \rho_{\psi}.
\end{align}
Therefore, $[\widehat{A}^\times: \det(\rho_\phi(\Gal_F))]=[\widehat{A}^\times: \rho_{\psi}(\Gal_F)]$ and hence Theorem~\ref{T:det image} follows from Theorem~\ref{T:Gekeler}(\ref{T:Gekeler ii}) with $\Delta:=-a_2$.

We now prove Proposition~\ref{P:det image}.  We have $\det \bbar\rho_{\phi,\aA} = \bbar\rho_{\psi,\aA}$ by (\ref{E:Weil pairing consequence}).   With $\Delta:=-a_2$, Theorem~\ref{T:Gekeler}(\ref{T:Gekeler i}) implies that $[(A/\aA)^\times: \det( \bbar\rho_{\phi,\aA}(\Gal_F))]$ divides $g$.  In particular, $\det(\bbar\rho_{\phi,\aA}(\Gal_F))=(A/\aA)^\times$ if $g=1$.

\section{Irreducibility}
\label{S:irreducibility}
Let $\phi\colon A\to F\{\tau\}$ be a Drinfeld $A$-module of rank $2$.    Suppose that $\lambda$ is a nonzero prime ideal of $A$ for which 
\begin{itemize}
\item
$\bbar\rho_{\phi,\lambda}\colon \Gal_F \to \GL_2(\FF_\lambda)$ is reducible,
\item
$\phi$ has stable reduction at $\lambda$.
\end{itemize}
After conjugating $\bbar\rho_{\phi,\lambda}$, we may assume that
\begin{align} \label{E:upper triangular chis}
\bbar\rho_{\phi,\lambda}(\sigma) = \left(\begin{smallmatrix} \chi_1(\sigma) & * \\ 0 & \chi_2(\sigma) \end{smallmatrix}\right)
\end{align}
for all $\sigma\in \Gal_F$, where $\chi_1,\chi_2\colon \Gal_F \to \FF_\lambda^\times$ are characters.   In this section, we will give a bound on the norm of $\lambda$. 

\begin{lemma} \label{L:basic unramified chi new}
Set $n:=(q-1)^2(q+1)$.
\begin{romanenum}
\item \label{L:basic unramified chi new i}
The characters $\chi_1^n$ and $\chi_2^n$ are both unramified at any nonzero prime ideal $\p\neq \lambda$ of $A$.
\item \label{L:basic unramified chi new ii}
One of the characters $\chi_1^n$ or $\chi_2^n$ is unramified at $\lambda$.
\end{romanenum}
\end{lemma}
\begin{proof}
Take any nonzero prime ideal $\p$ of $A$.   We shall view $\phi$ as being defined over $F_\p$ and hence $\bbar\rho_{\phi,\lambda}$, $\chi_1$ and $\chi_2$ are representations of $\Gal_{F_\p}$.  Let $I_\p$ be an inertia subgroup of $\Gal_{F_\p}$.

Suppose that $\p=\lambda$ and $\phi$ has good reduction.  
The cardinality of the group $\bbar\rho_{\phi,\lambda}(\Gal_F)$ divides $q^{\deg \lambda} (q^{\deg \lambda}-1)^2$ by (\ref{E:upper triangular chis}) and hence $\bbar\rho_{\phi,\lambda}(\Gal_F)$ cannot have a subgroup of cardinality $q^{2\deg \lambda }-1$.   Therefore, property (\ref{P:good inertia a}) of Proposition~\ref{P:good inertia} holds and this implies that $\chi_1$ or $\chi_2$ is unramified at $\lambda$.   Now suppose that $\p=\lambda$ and $\phi$ does not have good reduction.  By assumption on $\lambda$, the Drinfeld module $\phi$ has stable reduction of rank $1$.  Proposition~\ref{P:main Tate new}(\ref{P:main Tate new i}) implies that one of $\chi_1^{q-1}$ or $\chi_2^{q-1}$ is unramified at $\lambda$.  We have thus proved (\ref{L:basic unramified chi new ii}) since $q-1$ divides $n$.

We may now assume that $\p\neq \lambda$.    We have $\phi_t=t+a_1\tau+a_2\tau^2$ with $a_1\in F$ and $a_2\in F^\times$.  Define $m:=\min\{ v_\p(a_i)/(q^i-1): 1 \leq i \leq 2\}$ and take $1\leq j \leq 2$ maximal such that $v_\p(a_j)/(q^j-1)=m$.   

Let $K$ be the splitting field of $x^{q^{j}-1}=a_j$ over $F_\p$ and fix a root $b\in K$.  The extension $K/F_\p$ is finite and Galois.  The ramification index $e:=e(K/F_\p)$ of the extension $K/F_\p$ divides $q^j-1$.   Let $I_K$ be the inertia subgroup of $\Gal_K \subseteq \Gal_{F_\p}$.   For any $\sigma\in I_\p$, we have $\sigma^e \in I_K$.   It thus suffices to prove that $\chi_1^{n/e}(I_K)=1$ and $\chi_2^{n/e}(I_K)=1$.  

Let $\OO$ be the integral closure of $A_\p$ in $K$.  Let $\phi'\colon A\to K\{\tau\}$ be the Drinfeld module for which $\phi'_t=b \phi_t b^{-1}$.  The Drinfeld module $\phi'$ is isomorphic to $\phi$ over $K$.

Suppose $j=2$. Then $\phi'$ is defined over $\OO$ and has good reduction.  Since $\phi'$ has good reduction and $\p\neq \lambda$, we have $\bbar\rho_{\phi',\lambda}(I_K)=1$.   Therefore, $\chi_1(I_K)=1$ and $\chi_2(I_K)=1$.  Since $e$ divides $q^2-1$ and hence $n$, we have $\chi_1^{n/e}(I_K)=1$ and $\chi_2^{n/e}(I_K)=1$.

Suppose that $j=1$.  Then $\phi'$ is defined over $\OO$ and has stable reduction of rank $1$.    By Proposition~\ref{P:main Tate new}(\ref{P:main Tate new i}) and $\lambda\neq \p$, we have $\chi_1^{q-1}(I_K)=1$ and $\chi_2^{q-1}(I_K)=1$.  Since $n/e$ is divisible by $n/(q-1)=(q-1)(q+1)$, we have $\chi_1^{n/e}(I_K)=1$ and $\chi_2^{n/e}(I_K)=1$.
\end{proof}

For a monic polynomial $P(x)\in A[x]$ and a positive integer $n\geq 1$, we let $P^{(n)}(x)\in A[x]$ be the monic polynomial whose roots (with multiplicity), in some algebraic closure, are precisely the roots of $P(x)$ raised to the $n$-th power.

\begin{lemma} \label{L:riemann-hurwitz applied}
Let $n$ be a positive integer for which parts (\ref{L:basic unramified chi new i}) and (\ref{L:basic unramified chi new ii}) of Lemma~\ref{L:basic unramified chi new} hold.  Then there is a $\zeta\in \FF_\lambda^\times$ such that $P_{\phi,\p}^{(n)}(\zeta^{\deg \p})= 0$ in $\FF_\lambda$ for all nonzero prime ideals $\p\neq\lambda$ of $A$ for which $\phi$ has good reduction.
\end{lemma}
\begin{proof}
By assumption, there is an $i\in \{1,2\}$ such that $\chi_i^n$ is unramified at all nonzero prime ideals of $A$.

We claim that $\chi_i^n(\Gal(F^\sep/\FFbar_q(t))=1$.   Let $L$ be the minimal extension of $\FFbar_q(t)$ in $F^\sep$ for which $\chi_i^n(\Gal(F^\sep/L))=1$.   The extension $L/\FFbar_q(t)$ corresponds to a morphism $\pi\colon C\to \PP^1_{\FFbar_q}$ of smooth projective and irreducible curves over $\FFbar_q$.  From our assumptions on $\chi_i^n$, $\pi$ is unramified away from all points of $\PP^1_{\FFbar_q}$ except perhaps $\infty$.    The morphism $\pi$ has degree $N:=[L:\FFbar_q(t)]$ which is relatively prime to $q$.   Since $L/\FFbar_q(t)$ is tamely ramified, the Riemann--Hurwitz theorem implies that $2g-2=N(2\cdot 0-2)+ \sum_{i=1}^s (e_i-1)$, where $g$ is the genus of $C$ and the $e_i\geq 1$ are the ramification indices at the $s$ points of $C$ lying over $\infty$.  Since $\sum_{i=1}^s e_i=N$, we have $2g=2-N-s$.  Since $N$ and $s$ are positive integers and $g\geq 0$, we must have $N=1$.  This proves the claim.

From the claim $\chi_i^n$ factors through a cyclic Galois group $\Gal(\FF_{q^d} F/F)$ for some $d\geq 1$. So there is a $\zeta\in \FF_\lambda^\times$ for which $\chi_i^n(\Frob_\p)=\zeta^{\deg \p}$ for all nonzero prime ideals $\p$ of $A$.   If $\p\neq \lambda$ and $\phi$ has good reduction at $\p$, then $\chi_i^n(\Frob_\p)=\zeta^{\deg \p}$ is a root of $\det(xI- \bbar\rho_{\phi,\lambda}(\Frob_\p)^n)\equiv P^{(n)}_{\phi,\p}(x) \pmod{\lambda}$.
\end{proof}

\begin{prop} \label{L:irreducibility bound of deg lambda}
Let $n$ be a positive integer for which parts (\ref{L:basic unramified chi new i}) and (\ref{L:basic unramified chi new ii}) of Lemma~\ref{L:basic unramified chi new} hold.  Let $d\geq 1$ be an integer for which $\phi$ has good reduction at multiple nonzero prime ideals of $A$ with the same degree $d$.   Then $\deg \lambda\leq 2nd$.
\end{prop}
\begin{proof}
Let $\p_1$ and $\p_2$ be distinct prime ideals of $A$ with $\deg \p_1=\deg \p_2=d$ for which $\phi$ has good reduction.  We may assume that $\lambda\notin \{\p_1,\p_2\}$ since otherwise $\deg \lambda =d$ and the proposition is immediate.  For each $1\leq i \leq 2$, set $Q_i(x):= P^{(n)}_{\phi,\p_i}(x)\in A[x]$.

Let $r\in A$ be the resultant of the polynomials $Q_1(x)$ and $Q_2(x)$.  By our choice of $n$ and Lemma~\ref{L:riemann-hurwitz applied}, the polynomials $Q_1(x)$ and $Q_2(x)$ have a common root modulo $\lambda$ and hence $r\equiv 0 \pmod{\lambda}$.  

Let $L/F$ be the splitting field of $P_{\phi,\p_1}(x)$ and $P_{\phi,\p_2}(x)$.   Let $|\cdot|_\infty$ be an absolute value on $L$ that satisfies $|a|_\infty = q^{\deg a}$ for all nonzero $a\in A$.   For any root $\pi\in L$ of $P_{\phi,\p_i}(x)$, we have $|\pi|_\infty=N(\p_i)^{1/2}= q^{(\deg \p_i)/2}=q^{d/2}$, cf.~\cite[Theorem 4.12.8(5)]{MR1423131}.  So for any root $\pi\in L$ of $Q_i(x)$, we have $|\pi|_\infty=q^{nd/2}$.  Recall that the resultant $r$ is the product of $\pi_1-\pi_2$ as we vary over all roots $\pi_1$ and $\pi_2$ in $L$ of 
$Q_1(x)$ and $Q_2(x)$, respectively.  Therefore, $|r|_\infty\leq (q^{nd/2})^4=q^{2nd}$.

We claim that $Q_1(x)$ and $Q_2(x)$ do not have a common roots in $L$.   To the contrary suppose that $\pi \in L$ is a root of $Q_1(x)$ and $Q_2(x)$.   Take any $i \in \{1,2\}$.   From \cite[Theorem 4.12.8(1)]{MR1423131} applied to the reduction of $\phi$ modulo $\p_i$, there is a unique place of $F(\pi)$ for which $\pi$ has a zero and it lies over the place of $\p_i$.   We get a contradiction since $\p_1\neq \p_2$ and the claim follows.

From the claim, we have $r\neq 0$.  Since $r$ is a nonzero element of $A$ with $|r|_\infty\leq q^{2nd}$, we have $\deg r \leq 2nd$.  Since $r$ is nonzero and $r\equiv 0 \pmod{\lambda}$, we have $\deg \lambda \leq \deg r$.  Therefore, $\deg \lambda \leq \deg r \leq 2nd$.
\end{proof}

\section{Hilbert irreducibility}
\label{S:Hilbert irreducibility}

Fix an integer $r\geq 2$ and a nonzero ideal $\aA$ of $A$. For any $a=(a_1,\ldots, a_r) \in A^r$ with $a_r\neq 0$, we let $\phi(a)\colon A\mapsto F\{\tau\}$ be the Drinfeld $A$-module of rank $r$ for which $\phi(a)_t = t + a_1 \tau + \cdots + a_{r-1} \tau^{r-1} + a_r \tau^r$ and we let 
\[
\bbar\rho_{\phi(a),\aA}\colon \Gal_F\to\GL_r(A/\aA)
\]
be the corresponding Galois representation.   The goal of this section is to prove the following theorem which says that $\bbar\rho_{\phi(a),\aA}$ is surjective for ``most'' $a$.

\begin{thm} \label{T:HIT Drinfeld}
The set of $a\in A^r$ such that $a_r\neq 0$ and $\bbar\rho_{\phi(a),\aA}(\Gal_F)=\GL_r(A/\aA)$ has density $1$.
\end{thm}

\subsection{A version of Hilbert's irreducibility theorem} \label{S:new HIT}

Fix a positive integer $r$ and a nonempty open subscheme $U$ of $\AA^r_{F}$.   Consider a continuous and surjective representation
\[
\rho\colon \pi_1(U)\to G,
\]
where $\pi_1(U)$ is the \'etale fundamental group of $U$ and $G$ is a finite group.  Here we are suppressing the base point of our fundamental group and hence the representation $\rho$ is only determined up to conjugacy by an element in $G$.   

Take any point $u\in U(F)\subseteq F^r$.  Specialization by $u$ gives rise to a continuous representation
\[
\rho_u\colon \Gal_F \xrightarrow{u_*} \pi_1(U)\xrightarrow{\rho} G
\]
that is uniquely determined up to conjugacy in $G$.  In particular, the group $\rho_u(\Gal_F)\subseteq G$ is uniquely determined up to conjugacy in $G$.   We will show that $\rho_u(\Gal_F)=G$ for all $u\in U(F)\cap A^r$ away from a set of density $0$ after first proving an easy lemma.

\begin{lemma} \label{L:easy density}
Let $I$ be a nonzero ideal of $A$ and fix a subset $B \subseteq (A/I)^r$.  Then the set of $a\in A^r$ whose image modulo $I$ lies in $B$ has density $|B|/N(I)^r$.
\end{lemma}
\begin{proof}
Consider a positive integer $d$.  Define the reduction map $\varphi_d\colon A^r(d) \to (A/I)^r$; it is a homomorphism of finite additive groups.   By taking $d$ sufficiently large, we may assume that $\varphi_d$ is surjective.  Therefore, the cardinality of $\varphi_d^{-1}(b)$ is the same for all $b\in (A/I)^r$.  In particular, $|\varphi_d^{-1}(B)|/|A^r(d)|=|B|/|(A/I)^r|$ for all sufficiently large $d$.   The lemma is now immediate.
\end{proof}

\begin{thm} \label{T:new HIT}
The set of $u\in U(F)\cap A^r$ for which $\rho_u(\Gal_F)=G$ has density $1$.
\end{thm}
\begin{proof}
For a fixed algebraic closure $\bbar{F}$ of $F$, we define the group $G_g:=\rho(\pi_1(U_{\bbar{F}}))$; it is a closed and normal subgroup of $G$.  Let $F'/F$ be the minimal extension in $\bbar{F}$ for which $G_g=\rho(\pi_1(U_{F'}))$.   The extension $F'/F$ is Galois and we have a natural short exact sequence 
\[
1\to G_g \to G \xrightarrow{\pi} \Gal(F'/F) \to 1.
\]

Take any proper subgroup $H$ of $G$ and let $S$ be the set of $u\in U(F) \cap A^r$ for which $\rho_u(\Gal_F)$ is conjugate in $G$ to a subgroup of $H$.   We will prove that $S$ has density $0$.   This will prove the theorem since $G$ has only finitely many proper subgroups.   We have $\pi(\rho_u(\Gal_F))=\Gal(F'/F)$ for all $u\in U(F)$, so we may assume that $\pi(H)=\Gal(F'/F)$ since otherwise $S$ is empty.  We thus have $H\cap G_g\subsetneq G_g$ since $H$ is a proper subgroup of $G$.  Define $C:= \bigcup_{g\in G} g H g^{-1}$.   Since $G_g$ is a normal subgroup of $G$ and $\pi(H)=\Gal(F'/F)$, we have 
\[
C\cap G_g = \bigcup_{g\in G} g (H\cap G_g) g^{-1}=\bigcup_{g\in G_g} g (H\cap G_g) g^{-1}.
\]
Since $H \cap G_g$ is a proper subgroup of $G_g$, we have $C\cap G_g \subsetneq G_g$ by Jordan's lemma (\cite[Theorem~4']{MR1997347}).

There is a ring $R:=A[1/n]\subseteq F$ with a nonzero $n \in A$, an $R$-subscheme $\calU \subseteq \AA^r_{R}$, and a representation
\[
\varrho\colon \pi_1(\calU)\to G
\]
so that $\calU_F=U$ and base changing $\varrho$ by $F$ gives $\rho$.  Take any nonzero prime ideal $\p$ of $A$ that does not divide $n$.  We will also denote the prime ideal $\p R$ of $R$ by $\p$; we have $R/\p=A/\p=\FF_\p$.     For each $u\in \calU(\FF_\p)$, specialization gives a homomorphism $u_*\colon \Gal_{\FF_\p} \to \pi_1(\calU)$ and we denote by $\Frob_u$ the image of the $N(\p)$-power Frobenius.  Note that $\Frob_u$ in $\pi_1(\calU)$ is uniquely determined up to conjugacy.   Define the set 
\[
\Omega_\p:=\{ u \in \calU(\FF_\p) : \varrho(\Frob_u) \in G-C\}.
\]
Now take any $u\in S$.  We have $\rho_u(\Gal_F)\subseteq C$.  If $u$ modulo $\p$ lies in $\calU(\FF_\p)$, then $\rho_u$ is unramified at $\p$ and $\rho_u(\Frob_\p)$ lies in the same conjugacy class of $G$ as $\varrho(\Frob_u)$.   So if $u$ modulo $\p$ lies in $\Omega_\p$, then $\rho_u(\Frob_\p)$ lies in $G-C$ which contradicts that $\rho_u(\Gal_F)\subseteq C$.  Therefore, the image of $S$ modulo $\p$ lies in $\FF_\p^r-\Omega_\p$. By Lemma~\ref{L:easy density}, we have
\[
\bbar\delta(S) \leq \prod_{\p \in \calP} \frac{|\FF_\p^r-\Omega_\p|}{|\FF_\p^r|} = \prod_{\p \in \calP} \Big(1 - \frac{|\Omega_\p|}{N(\p)^r} \Big),
\]
where $\calP$ is any finite set of nonzero prime ideals of $A$ that do not divide $n$.  So to show that $S$ has denisty $0$ it suffices to prove that there is a positive constant $c<1$ such that $1-|\Omega_\p|/N(\p)^r<c$ for infinitely many prime ideals $\p$ of $A$.

Now take a nonzero prime ideal $\p$ of $A$ that splits completely in $F'$ and does not divide $n$.    Our assumption that $\p$ splits completely in $F'$ implies that $\varrho(\pi_1(\calU_{\FF_\p}))\subseteq G_g$.

We claim that $\varrho(\pi_1(\calU_{\FFbar_\p}))=G_g$ for all but finitely many such $\p$.  Let $R'$ be the integral closure of $R$ in $F'$.   We can base change $\varrho$ to get a surjective representation $\varrho'\colon \pi_1(\calU_{R'})\to G_g$.  To prove the claim, it suffices to show that $\varrho'(\pi_1((\calU_{R'})_{\FF_\P}))=\varrho'(\pi_1(\calU_{\FF_\P}))$ equals $G_g$ for all but finitely many nonzero prime ideals $\P$ of $R'$.    The representation $\varrho'$ corresponds to an \'etale cover $Y\to  \calU_{R'}$ of $R'$-schemes so that $Y_{F'}$ and $(\calU_{R'})_{F'}=\calU_{F'}$ are both geometrically irreducible.    For nonzero prime ideals $\P$ of $R'$, we have an \'etale cover $Y_{\FF_\P} \to  \calU_{\FF_\P}$ of degree $|G_g|$.   The claim follows from $Y_{\FF_\P}$ and $\calU_{\FF_\P}$ being geometrically irreducible for all but finitely many $\P$, cf. \cite[9.7.8]{MR217086}.

So by excluding a finite number of $\p$, we may assume that $\varrho(\pi_1(\calU_{\FF_\p}))=\varrho(\pi_1(\calU_{\FFbar_\p}))=G_g$.  Then an explicit equidistribution result like 
\cite[Theorem~3]{MR4285725} implies that 
\[
|\Omega_\p|=
|\{u \in \calU(\FF_\p) : \varrho(\Frob_u) \in G_g-(C\cap G_g)\}| = 
\tfrac{|G_g-(C \cap G_g)|}{|G_g|} N(\p)^r + O(N(\p)^{r-1/2}),
\]
where the implicit constant does not depend on $\p$.  To apply \cite[Theorem~3]{MR4285725} we are using that the ``complexity'' of $Y_{\FF_\P}\to \calU_{\FF_\P}$ can be bounded independent of the nonzero prime ideal $\P$ of $R'$ since they arise from a single morphism $Y\to \calU_{R'}$.

Therefore, $1-|\Omega_\p|/N(\p)^r = {|C \cap G_g|}/{|G_g|} + O(N(\p)^{-1/2})$.  Since $C\cap G_g \subsetneq G_g$, there is a constant $c<1$ such that $1-|\Omega_\p|/N(\p)^r<c$ holds for all but finitely many prime ideals $\p$ of $A$ that split completely in $F'$.  The result follows since there are infinitely many prime ideals $\p$ of $A$ that split completely in $F'$.
\end{proof}

\subsection{Proof of Theorem~\ref{T:HIT Drinfeld}}

Define the $F$-algebra $R:= F[b_1,\ldots, b_r, 1/b_r]$, where $b_1,\ldots, b_r$ are indeterminant variables over $F$.     Let $\phi\colon A\to R\{\tau\}$, $\alpha\mapsto \phi_\alpha$ be the homomorphism of $\FF_q$-algebras for which 
\begin{align}\label{E:generic Drinfeld}
\phi_t= t  + b_1 \tau + \cdots + b_{r-1} \tau^{r-1} + b_r \tau^r;
\end{align}
it is a Drinfeld $A$-module of rank $r$ over the scheme $U:=\Spec R$.  Note that $U$ is a nonempty open subscheme of $\AA_F^r=\Spec F[b_1,\ldots, b_r]$.   The $\aA$-torsion of $\phi$ gives rise to a representation
\[
\bbar\rho_{\phi,\aA} \colon \pi_1(U)\to \GL_r(A/\aA)
\]
like before.

Take any $a=(a_1,\ldots, a_r)\in U(F)\subseteq F^r$.   We have $a_r\neq 0$, so (\ref{E:generic Drinfeld}) with $b_i$ replaced by $a_i$ gives a Drinfeld $A$-module $\phi(a)\colon A\to F\{\tau\}$ of rank $r$.  Specializing $\bbar\rho_{\phi,\aA}$ at $a$  gives a representation $\Gal_F \to \GL_r(A/\aA)$ that is isomorphic to $\bbar\rho_{\phi(a),\aA}$. Theorem~\ref{T:HIT Drinfeld} will thus follow from Theorem~\ref{T:new HIT} and the following lemma.

\begin{lemma}
We have $\bbar\rho_{\phi,\aA}(\pi_1(U))=\GL_r(A/\aA)$.
\end{lemma}
\begin{proof}

Let $V$ be the closed subvariety of $U$ defined by the equation $b_r=1$, i.e., corresponding to the prime ideal $\P=(b_r-1)$ of $R$.  Specialization $\bbar\rho_{\phi,\aA}$ at $\P$ gives a representation
\[
\varrho\colon \pi_1(V)\to \GL_r(A/\aA).
\]
It suffices to prove that $\varrho$ is surjective.    The representation $\varrho$ agrees with $\bbar\rho_{\psi,\aA}$ where $\psi\colon A\to R'\{\tau\}$, $\alpha\mapsto \psi_\alpha$ is the Drinfeld $A$-module with $\psi_t= t  + b_1 \tau + \cdots + b_{r-1} \tau^{r-1} + \tau^r$ and $R'=F[b_1,\ldots, b_{r-1}]$.  

Set $K:=F(b_1,\ldots, b_{r-1})$.  Viewing $\psi$ as a Drinfeld $A$-module over $K$, it thus suffices to show that $\Gal(K(\psi[\aA])/K) \cong \GL_r(A/\aA)$, where $\psi[\aA]$ is the $\aA$-torsion arising from $\psi$ in a fixed separable closure of $K$.   In \cite[Theorem 6]{MR3425215}, Breuer proves that $\Gal(K(\psi[\aA])/K) \cong \GL_r(A/\aA)$ which had been earlier conjectured by Abhyankar.
\end{proof}

\section{Sieving} \label{S:sieving}

Let $\calB$ be the set of $a=(a_1,a_2) \in A^2$ with $a_2\neq 0$ for which the following hold:
\begin{itemize}
\item $\rho_{\phi(a),\lambda}(\Gal_F)=\GL_2(A_\lambda)$ for all nonzero prime ideals $\lambda$ of $A$, 
\item the commutator subgroup of $\rho_{\phi(a)}(\Gal_F) \subseteq \GL_2(\widehat{A})$ is equal to $[\GL_2(\widehat{A}),\GL_2(\widehat{A})]$.
\end{itemize}

The main goal of this section is to prove the following theorem.   We will use it in \S\ref{SS:main new proof -main} to quickly prove Theorem~\ref{T:main new} in the case $q\neq 2$.   

\begin{thm} \label{T:almost main}
The set $\calB$ has density $1$.
\end{thm}

\subsection{Proof of Theorem~\ref{T:almost main}}

Fix an integer $m\geq 2$.  
\begin{itemize}
\item
Let $\calR$ be the set of $(a_1,a_2)\in A^2$ for which there are at least two distinct nonzero prime ideals $\p$ of $A$ that satisfy $\deg \p >1$, $v_\p(a_1)=0$ and $v_\p(a_2)=1$.  
\item
Let $\calS_m$ be the set of $(a_1,a_2)\in A^2$ for which $a_1\not\equiv 0\pmod{\p}$ or $a_2\not\equiv 0\pmod{\p}$ for all nonzero prime ideals $\p$ of $A$ with $\deg \p > m$.
\item
Let $\calT_m$ be the set of $(a_1,a_2)\in A^2$ for which there are two distinct prime ideals $\p_1$ and $\p_2$ of $A$ of the same degree $d\leq m/(2(q-1)^2(q+1))$ for which $a_2\not\equiv 0 \pmod{\p_1}$ and $a_2\not\equiv 0\pmod{\p_2}$.
\item
Let $\calU_m$ be the set of $(a_1,a_2)\in A^2$ with $a_2\neq 0$ for which  $\bbar\rho_{\phi(a),\lambda^2}(\Gal_F)=\GL_2(A/\lambda^2)$ for all nonzero prime ideals $\lambda$ of $A$ with $\deg \lambda \leq m$.  
\end{itemize}

\begin{lemma} \label{L:sieving lambda-adic images}
For any $a\in \calR \cap \calS_m \cap \calT_m \cap \calU_m$ and nonzero prime ideal $\lambda$ of $A$, we have $\rho_{\phi(a),\lambda}(\Gal_F)=\GL_2(A_\lambda)$.
\end{lemma}
\begin{proof}
Take any $a=(a_1,a_2) \in \calR\cap \calS_m \cap \calT_m \cap \calU_m$ and any nonzero prime ideal $\lambda$ of $A$.   We have $a_2\neq 0$.   Define $G:=\rho_{\phi(a),\lambda}(\Gal_F)$; it is a closed subgroup of $\GL_2(A_\lambda)$.  With $R:=A_\lambda$, we will show that $G$ satisfies the conditions of Proposition~\ref{P:full GL2R criterion}.    Once we have verified that the conditions hold, Proposition~\ref{P:full GL2R criterion} will imply that $G=\GL_2(A_\lambda)$.  

Since $a\in \calR$, there is a nonzero prime ideal $\p\neq \lambda$ of $A$ such that $v_\p(a_1)=0$ and $v_\p(a_2)=1$.   In particular, $\phi(a)$ has stable reduction of rank $1$ at $\p$.    Define the $j$-invariant $j_{\phi(a)}:=a_1^{q+1}/a_2\in F$.  We have $v_\p(j_{\phi(a)})=-1$.     Since $v_\p(a_2)=1$ and $\p\neq \lambda$, $\det(G)=A_\lambda^\times$ by Proposition~\ref{P:det image}.   This verifies condition (\ref{P:full GL2R criterion a}) of Proposition~\ref{P:full GL2R criterion}.

 Let $I_{\p}$ be an inertia subgroup of $\Gal_F$ for the prime $\p$.  Take any $i\geq 1$.  Since $\p\neq \lambda$ and the denominator of $v_\p(j_{\phi(a)})/N(\lambda^i)=-1/N(\lambda)^i$ is $N(\lambda)^i$,  Proposition~\ref{P:main Tate new} implies that there is a subgroup of $\left\{ \left(\begin{smallmatrix} 1 & b \\0 & 1\end{smallmatrix}\right) : b\in A/\lambda^i \right\}$ of order $N(\lambda)^i$ that is conjugate in $\GL_2(A/\lambda^i)$ to a subgroup of $\bbar\rho_{\phi(a),\lambda^i}(I_\p)$.  So after choosing an appropriate basis for $\phi[\lambda^i]$ for all $i\geq 1$, we may assume that 
  \begin{align} \label{E:Tate consequence}
\left\{ \left(\begin{smallmatrix} 1 & b \\0 & 1\end{smallmatrix}\right) : b\in A/\lambda^i \right\}\subseteq  \bbar\rho_{\phi(a),\lambda^i}(\Gal_F).
  \end{align}  
We thus have $\left(\begin{smallmatrix} 1 & 1 \\0 & 1\end{smallmatrix}\right) \in G$ and hence condition (\ref{P:full GL2R criterion e}) of Proposition~\ref{P:full GL2R criterion} holds.   With $i=2$, (\ref{E:Tate consequence}) implies that condition (\ref{P:full GL2R criterion c}) of Proposition~\ref{P:full GL2R criterion} holds.  If $N(\lambda)=2$, then $\bbar\rho_{\phi(a),\lambda^2}(\Gal_F)=\GL_2(A/\lambda^2)$ since $a\in \calU_m$ and $m\geq 2$; this shows that condition (\ref{P:full GL2R criterion d}) of Proposition~\ref{P:full GL2R criterion} holds.
  
It remains to verify that condition (\ref{P:full GL2R criterion b}) of Proposition~\ref{P:full GL2R criterion} holds, i.e., show that $\bbar\rho_{\phi(a),\lambda}(\Gal_F)=\GL_2(\FF_\lambda)$.   Since $a\in \calU_m$, we may assume that $\deg \lambda > m$.   Since $a\in \calS_m$, we have $a_1\not\equiv 0\pmod{\lambda}$ or $a_2\not\equiv 0 \pmod{\lambda}$.  Therefore, $\phi(a)$ has stable reduction at $\lambda$.  Since $\bbar\rho_{\phi(a),\lambda}(\Gal_F)$ contains a subgroup of order $N(\lambda)$ by (\ref{E:Tate consequence}), Proposition~\ref{P:group contains SL2 condition} implies that $\bbar\rho_{\phi(a),\lambda}(\Gal_F)\supseteq \SL_2(\FF_\lambda)$ or $\bbar\rho_{\phi(a),\lambda}$ is reducible.  
 
 Suppose that $\bbar\rho_{\phi(a),\lambda}$ is reducible.   Since $a\in \calT_m$, there are two distinct prime ideals $\p_1$ and $\p_2$ of $A$ of the same degree $d \leq m/(2(q-1)^2(q+1))$ for which $\phi(a)$ has good reduction at both $\p_1$ and $\p_2$.   By Lemmas~\ref{L:basic unramified chi new} and \ref{L:irreducibility bound of deg lambda} with $n:=(q-1)^2(q+1)$, we have $\deg \lambda \leq 2 (q-1)^2(q+1) d\leq m$.  This is a contradiction since $\deg \lambda>m$.   
 
 Therefore, $\bbar\rho_{\phi(a),\lambda}$ is irreducible and hence $\bbar\rho_{\phi(a),\lambda}(\Gal_F)\supseteq \SL_2(\FF_\lambda)$.  Since $\det(G)=A_\lambda^\times$, we deduce that $\det(\bbar\rho_{\phi(a),\lambda}(\Gal_F))=\FF_\lambda^\times$ and hence $\bbar\rho_{\phi(a),\lambda}(\Gal_F)= \GL_2(\FF_\lambda)$.
\end{proof}

\begin{lemma}  \label{L:sieving adelic images}
Take any $a\in \calR \cap \calS_m \cap \calT_m \cap \calU_m$.  Then the commutator subgroup of $\rho_\phi(\Gal_F) \subseteq \GL_2(\widehat{A})$ is equal to $[\GL_2(\widehat{A}),\GL_2(\widehat{A})]$.
\end{lemma}
\begin{proof}
Define $G:=\rho_\phi(\Gal_F)$; it is a closed subgroup of $\GL_2(\widehat{A})$.   We will now verify that $G$ satisfies all the conditions of Theorem~\ref{T:criterion for computing commutator}.   For any nonzero prime ideal $\lambda$ of $A$, we have $G_\lambda = \GL_2(A_\lambda)$ by Lemma~\ref{L:sieving lambda-adic images}.  This verifies condition (\ref{T:criterion for computing commutator a}) of Theorem~\ref{T:criterion for computing commutator}.

Now take any distinct nonzero prime ideals $\lambda_1$ and $\lambda_2$ of $A$.   Since $a\in \calR$, there is a nonzero prime ideal $\p$ of $A$ such that $v_\p(a_1)=0$, $v_\p(a_2)=1$, and $\deg \p>1$.   In particular, $\phi(a)$ has stable reduction of rank $1$ at $\p$.    We have $v_\p(j_{\phi(a)})=-1$.     

 Let $I_{\p}$ be an inertia subgroup of $\Gal_F$ for the prime $\p$.  Take any $i\geq 1$.  The denominator of $v_\p(j_{\phi(a)})/N(\lambda_1 \lambda_2)=-1/N(\lambda_1 \lambda_2)$ is $N(\lambda_1)N( \lambda_2)$ and hence $\bbar\rho_{\phi(a),\lambda_1\lambda_2}(I_\p)$ contains a subgroup of cardinality $N(\lambda_1)N(\lambda_2)$ by Proposition~\ref{P:main Tate new}(\ref{P:main Tate new ii}).  In particular, the image of $G$ modulo $\lambda_1\lambda_2$ contains a subgroup of cardinality $N(\lambda_1)N(\lambda_2)$.  This verifies that condition (\ref{T:criterion for computing commutator b}) of Theorem~\ref{T:criterion for computing commutator} holds.
 
 Now suppose that $N(\lambda_1)=N(\lambda_2)=2$.  We have $\p\notin \{\lambda_1,\lambda_2\}$ since $\deg \p>1$. Take any $i\geq 1$.  Proposition~\ref{P:main Tate new} implies that $\bbar\rho_{\phi(a),\lambda_1^i\lambda_2^i}(I_\p)$ is conjugate in $\GL_2(A/(\lambda_1^i\lambda_2^i))$ to a subgroup of $\left\{ \left(\begin{smallmatrix} 1 & b \\0 & 1\end{smallmatrix}\right) : b\in A/(\lambda_1^i \lambda_2^i) \right\}$ and that $\bbar\rho_{\phi(a),\lambda_1^i\lambda_2^i}(I_\p)$ has cardinality divisible by $N(\lambda_1^i\lambda_2^i)$.   So after choosing  appropriate bases for $\phi[\lambda_1^i \lambda_2^i]$ for all $i\geq 1$, we may assume that 
  \begin{align*} 
\left\{ \left(\begin{smallmatrix} 1 & b \\0 & 1\end{smallmatrix}\right) : b\in A_{\lambda_1\lambda_2} \right\}\subseteq  \rho_{\phi(a),\lambda_1\lambda_2}(\Gal_F).
  \end{align*}  
 This verifies that condition (\ref{T:criterion for computing commutator c}) of Theorem~\ref{T:criterion for computing commutator} holds.
 
Now suppose that $q\in \{2,3\}$ and let $\aA$ be the ideal that is the product of the prime ideals of $A$ of degree $1$.  We have $\deg \p > 1$.   Since $v_\p(a_2)=1$ and $\p\nmid \aA$, we have $\det(\bbar\rho_{\phi,\aA^i}(\Gal_F))=(A/\aA^i)^\times$ for all $i\geq 1$ by Proposition~\ref{P:det image}.  Condition (\ref{T:criterion for computing commutator d}) of Theorem~\ref{T:criterion for computing commutator} is now immediate.

We have verified the conditions of Theorem~\ref{T:criterion for computing commutator} and hence $G$ and $\GL_2(\widehat{A})$ have the same commutator subgroup.
\end{proof}

By Lemmas~\ref{L:sieving lambda-adic images} and \ref{L:sieving adelic images}, we have an inclusion
$\calR \cap \calS_m \cap \calT_m \cap \calU_m \subseteq \calB$ and hence
\begin{align}\label{E:calB inclusion}
A^2 - \calB \subseteq (A^2 -\calR) \cup (A^2 -\calS_m) \cup (A^2 - \calT_m) \cup (A^2 -\calU_m).
\end{align}
We now bounds the upper densities of the sets in the right-hand side of (\ref{E:calB inclusion}).

\begin{lemma} \label{L:density R}
We have $\delta(A^2-\calR)=0$.
\end{lemma}
\begin{proof}
Recall that the reciprocal of the zeta function of $\AA^1_{\FF_q}=\Spec A$ is the power series $\prod_{\p}(1-T^{\deg \p})$, where the product is over nonzero prime ideals of $A$, and it is equal to $1-qT$.  By considering $T=1/q$, and hence $T^{\deg \p}=1/N(\p)$, we find that $\lim_{d\to + \infty}\prod_{\p, \, \deg\p\leq d} (1-\tfrac{1}{N(\p)})=0$.   We can choose two disjoint sets $\calP_1$ and $\calP_2$ of nonzero prime ideals of $A$ with degree at least $2$ such that 
\begin{align} \label{E:euler product}
\lim_{d\to + \infty}\prod_{\p\in \calP_i, \, \deg\p\leq d} (1-\tfrac{1}{N(\p)})=0
\end{align}
for $1\leq i \leq 2$. 

Take any $1\leq i \leq 2$ and let $S_i$ be the set of $(a_1,a_2)\in A^2$ such that $(v_\p(a_1),v_\p(a_2))\neq (0,1)$ for all $\p\in \calP_i$.  Take any $\p\in \calP_i$ and let $\Omega_\p$ be the set of $(b_1,b_2)\in (A/\p^2)^2$ for which $b_1\not\equiv 0\pmod{\p}$, $b_2\equiv 0 \pmod{\p}$ and $b_2\not\equiv0 \pmod{\p^2}$.   Define 
\[
\alpha_\p:=\tfrac{|(A/\p^2)^2-\Omega_\p|}{|(A/\p^2)^2|} = 1-\tfrac{|\Omega_\p|}{|(A/\p^2)^2|}= 1 - (1-\tfrac{1}{N(\p)})(\tfrac{1}{N(\p)}-\tfrac{1}{N(\p)^2}) \leq (1-\tfrac{1}{N(\p)})(1+\tfrac{c}{N(\p)^2}),
\]
where $c\geq 1$ is an absolute constant.   Note that the image of $S_i$ modulo $\p^2$ lies in $(A/\p^2)^2-\Omega_\p$.  For any positive integer $d$, Lemma~\ref{L:easy density} implies that
\[
\bbar\delta(S_i) \leq \prod_{\p\in \calP_i \, \deg\p\leq d} \alpha_\p \leq \prod_{\p \in \calP_i, \, \deg \p \leq d} (1-\tfrac{1}{N(\p)})(1+\tfrac{c}{N(\p)^2}).
\]
Using the zeta function of $\AA^1_{\FF_q}$, we find that $\prod_{\p, \, \deg \p \leq d} (1+\tfrac{c}{N(\p)^2})$ can be bounded independent of $d$.  So there is a positive constant $C$ such that 
\[
\bbar\delta(S_i) \leq C \prod_{\p \in \calP_i, \, \deg \p \leq d} (1-\tfrac{1}{N(\p)})
\]
holds for any $d\geq 1$.  By (\ref{E:euler product}), we deduce that $\bbar\delta(S_i)=0$ and hence $\delta(S_i)=0$.

We have $\calR\supseteq A^2-(S_1 \cup S_2)$ since the sets $\calP_1$ and $\calP_2$ are disjoint.  Since $\delta(S_1)=\delta(S_2)=0$, we conclude that $\delta(\calR)=1$.
\end{proof}

\begin{lemma} \label{L:density S}
For any $\varepsilon>0$, we have $\bbar{\delta}(A^2-\calS_m)<\varepsilon$ for all large enough $m\geq 1$.
\end{lemma}
\begin{proof}
Define $\calC:=A^2-\calS_m$, i.e., the set of $(a_1,a_2)\in A^2$ for which $a_1\equiv 0 \pmod{\p}$ and $a_2\equiv 0 \pmod{\p}$ for some prime ideal $\p$ of $A$ with $\deg \p >m$.

Fix an integer $d>m$.  Let $\calP(d)$ be the set of $(a_1,a_2)\in A^2$ with $\deg(a_1)\leq d$ and $\deg(a_2)\leq d$.    Define $\calC(d):= \calC \cap \calP(d)$.   

Take any $(a_1,a_2) \in \calP(d) -\{(0,0)\}$ with $a_1\equiv 0 \pmod{\p}$ and $a_2 \equiv 0 \pmod{\p}$ for some nonzero prime ideal $\p$ of $A$.   Fix an $1\leq i \leq 2$ for which $a_i\neq 0$.    We have $a_i\equiv 0\pmod{\p}$ and hence $\deg \p \leq \deg a_i \leq d$.  

Therefore,
\begin{align} \label{E:calC bound series}
|\calC(d)| \leq \sum_{\p,\, m<\deg \p \leq d} \beta_\p(d),
\end{align}
where $\beta_\p(d)$ is the number of $(a_1,a_2)\in \calP(d)$ such that $a_1\equiv 0 \pmod{\p}$ and $a_2 \equiv 0 \pmod{\p}$. Take any nonzero prime ideal $\p$ of $A$ with $m<\deg(\p)\leq d$.   
  The reduction modulo $\p$ map $\calP(d)\to \FF_\p^2$ is a surjective homomorphism of $\FF_q$-vector spaces and hence the kernel has dimension $2(d+1)-2\deg\p$.  Therefore, $\beta_\p(d)=q^{2(d+1-\deg\p)}$.  Using (\ref{E:calC bound series}), we deduce that
\[
\frac{|\calC(d)|}{|\calP(d)|} \leq \sum_{\p,\, m< \deg \p \leq d} q^{-2\deg \p} \leq \sum_{\p,\, \deg \p >m} q^{-2\deg \p},
\]
where we note that the series converges absolutely.   So by taking $m\geq 1$ large enough, we will have $|\calC(d)|/|\calP(d)|<\varepsilon$ for all $d>m$.  Therefore, $\bbar\delta(A^2-\calS_m)=\bbar\delta(\calC)<\varepsilon$.
\end{proof}

\begin{lemma} \label{L:density T}
For any $\varepsilon>0$, we have $\bbar{\delta}(A^2-\calT_m) < \varepsilon$ for all large enough $m\geq 1$.   
\end{lemma}
\begin{proof}
Let $d(m)$ be the largest integer for which $d(m)\leq m/(2(q-1)^2(q+1))$.  By taking $m$ large enough, we may assume that there are distinct nonzero prime ideals $\p_1$ and $\p_2$ of $A$ with $\deg \p_1 = \deg \p_2 =d(m)$.  Let $\Omega$ be the set of $b\in A/(\p_1\p_2)$ for which $b\equiv 0 \pmod{\p_1}$ or $b\equiv 0 \pmod{\p_2}$.  The image of $A^2-\calT_m $ modulo $\p_1\p_2$ lies in $\Omega$.  By Lemma~\ref{L:easy density}, we have $\bbar\delta(A^2-\calT_m ) = |\Omega|/|A/(\p_1\p_2)| \leq 1/N(\p_1) + 1/N(\p_2)=2/q^{d(m)}$.  Since $d(m)\to \infty$ as $m\to \infty$,  we conclude that $\bbar\delta(A^2-\calT_m )<\varepsilon$ for large enough $m$.
\end{proof}

\begin{lemma} \label{L:density U}
We have have $\delta(A^2-\calU_m)=0$.
\end{lemma}
\begin{proof}
This follows from Theorem~\ref{T:HIT Drinfeld}.
\end{proof}

From the inclusion (\ref{E:calB inclusion}), we have
\[
\bbar\delta(A^2-\calB)\leq  \bbar\delta(A^2-\calR) +\bbar\delta(A^2-\calS_m) +\bbar\delta(A^2-\calT_m)+\bbar\delta(A^2-\calU_m).
\]
Take any $\varepsilon>0$.   Using Lemmas~\ref{L:density R}, \ref{L:density S}, \ref{L:density T} and \ref{L:density U}, we deduce that $\bbar\delta(A^2-\calB)<\varepsilon$ for all sufficiently large integers $m$.  Since $\varepsilon>0$ was arbitrary, we have $\bbar\delta(A^2-\calB)=0$ and hence $\delta(A^2-\calB)=0$.  Therefore, $\calB$ has density $1$.

\subsection{Proof of Theorem~\ref{T:main new} when $q\neq 2$} \label{SS:main new proof -main}

We assume throughout that $q\neq 2$.  

Take any $a\in \calB$. We have $\rho_{\phi(a),\lambda}(\Gal_F)=\GL_2(A_\lambda)$ for all nonzero prime ideals $\lambda$ of $A$ by our definition of $\calB$ and hence $a\in S_3$.    	We have $[\GL_2(\widehat{A}),\GL_2(\widehat{A})]=\SL_2(\widehat{A})$ by Proposition~\ref{P:commutator GL2(R)} and our assumption that $q\neq 2$.    By our definition of $\calB$, the commutator subgroup of $\rho_{\phi(a)}(\Gal_F)$ is $\SL_2(\widehat{A})$ and in particular we have $\rho_{\phi(a)}(\Gal_F)\supseteq \SL_2(\widehat{A})$.   The index $[\GL_2(\widehat{A}): \rho_{\phi(a)}(\Gal_F)]$ thus agrees with $[\widehat{A}^\times: \det(\rho_{\phi(a)}(\Gal_F))]$ which divides $q-1$ by Theorem~\ref{T:Gekeler}(\ref{T:Gekeler ii}).  Therefore, $a \in S_2$.    Since $a$ was an arbitrary element of $\calB$, we have $S_2\supseteq \calB$ and $S_3\supseteq \calB$.  Theorem~\ref{T:almost main}  implies that $S_2$ and $S_3$ have density $1$.

Let $\calP$ be the set of $(a_1,a_2) \in A^2$ with $a_2\neq 0$ for which the leading coefficient of the polynomial $(-1)^{\deg a_2+1} a_2$ generates the cyclic group $\FF_q^\times$.  The set $\calP$ has positive density; moreover, it has density $\varphi(q-1)/(q-1)$, where $\varphi$ is Euler's totient function.  For $a\in \calP$, we have $[\widehat{A}^\times: \det(\rho_{\phi(a)}(\Gal_F))]=1$ by Theorem~\ref{T:det image}.   

We have $S_1 \supseteq \calP\cap \calB$ since $\rho_{\phi(a)}(\Gal_F)\supseteq \SL_2(\widehat{A})$ and $\det(\rho_{\phi(a)}(\Gal_F))=\widehat{A}^\times$ for all $a\in \calP \cap \calB$.     We have $\delta(\calP\cap\calB)=\delta(\calP)>0$ since $\calB$ has density $1$ by Theorem~\ref{T:almost main}.  Therefore, $S_1$ has a subset with positive density.

\section{$q=2$ and wild ramification at $\infty$} \label{S:wild}

Throughout \S\ref{S:wild}, we assume that $q=2$.  The goal of this section is to give a condition that ensures $\rho_\phi(\Gal_F)$ equals $\GL_2(\widehat{A})$ assuming it contains its commutator subgroup. Let $v_\infty\colon F^\times \twoheadrightarrow \ZZ$ be the valuation satisfying $v_\infty(a)=-\deg(a)$ for all nonzero $a\in A$.

\begin{prop} \label{P:final char 2 conditions}
Let $\phi\colon A\to F\{\tau\}$ be a Drinfeld $A$-module of rank $2$ for which $v_\infty(j_\phi)$ is odd and $v_\infty(j_\phi)\leq -5$.  Then the homomorphism
\[
\Gal_F \to \GL_2(\widehat{A})/[\GL_2(\widehat{A}),\GL_2(\widehat{A})]
\]
obtained by composing $\rho_\phi$ with the obvious quotient map is surjective.
\end{prop}

\subsection{Maximal abelian quotient} \label{SS:maximal abelian quotient}

For each $i\in \FF_2$, define the prime ideal $\lambda_i=(t+i)$ of $A$.  Define the homomorphism
\[
\gamma_i \colon \GL_2(\widehat{A}) \to \GL_2(\FF_{\lambda_i})=\GL_2(\FF_2) \to \GL_2(\FF_2)/[\GL_2(\FF_2),\GL_2(\FF_2)] \cong \{\pm 1\},
\]
where we are composing reduction modulo $\lambda_i$ with the quotient map.  We obtain a surjective continuous homomorphism 
\[
\beta\colon \GL_2(\widehat{A}) \to  \widehat{A}^\times \times \{\pm 1\} \times \{\pm 1\},\quad B\mapsto (\det(B), \gamma_0(B), \gamma_1(B)).
\]

\begin{lemma} \label{L:commutator q=2}
The kernel of $\beta$ is $[\GL_2(\widehat{A}), \GL_2(\widehat{A})]$.
\end{lemma}
\begin{proof}
We have $\GL_2(\widehat{A})=\prod_{\lambda} \GL_2(A_\lambda)$, where the product is over the nonzero prime ideals of $A$.  Therefore, the commutator subgroup of $\GL_2(\widehat{A})$ is equal to $\prod_{\lambda} [\GL_2(A_\lambda),\GL_2(A_\lambda)]$.   Using Proposition~\ref{P:commutator GL2(R)}, we obtain a description of $[\GL_2(\widehat{A}),\GL_2(\widehat{A})]$ that agrees with the kernel of $\beta$.
\end{proof}

\subsection{Cubic polynomials} \label{SS:cubic}

Consider a separable polynomial $f(x)=x^3+bx+c \in K[x]$, where $K$ is a field.  Let $r_1$, $r_2$ and $r_3$ be the distinct roots of $f(x)$ in some separable closure of $K$ and define the splitting field $K':=K(r_1,r_2,r_3)$.   Using the numbering of the $r_i$, we have an injective homomorphism $\iota\colon \Gal(K'/K)\hookrightarrow \mathfrak{S}_3$ to the symmetric group on $3$ letters.   Let $\varepsilon \colon \Gal(K'/K)\to \{\pm 1\}$ be the homomorphism obtained by composing $\iota$ with the parity character.   

Let $L/K$ be the subfield of $K'$ fixed by the kernel of $\varepsilon$.   We want to explicitly describe the extension $L/K$.   When $K$ has odd characteristic, $L$ is obtained by adjoining to $K$ a square root of the discriminant of $f(x)$.   This is not good enough for our application which concerns fields of characteristic $2$.  The following material comes from \cite{conrad} and is straightforward to prove.
 
Define the polynomial
\[
R_2(x):=(x-(r_1^2 r_2 + r_2^2 r_3 +r_3^2 r_1))(x-(r_2^2 r_1 + r_1^2 r_3 +r_3^2 r_2)).
\]
When expanded out, the coefficients of $R_2(x)$ are symmetric polynomials in $r_1,r_2,r_3$ and hence are polynomials in $b$ and $c$.   A direct computation shows that
\[
R_2(x)=x^2-3c x +(b^3+9c^2).
\]
One can then verify that the $f(x)$ and $R_2(x)$ have the same discriminant.  In particular, $R_2(x)$ is separable since $f(x)$ is separable.    Note that $r_1^2r_2+r_2^2r_3+r_3^2 r_1$ and $r_2^2 r_1 + r_1^2 r_3 +r_3^2 r_2$ are both fixed by any even permutation of the $r_i$ but are swapped by any odd permutation of the $r_i$.  Therefore, $L$ is the splitting field of $R_2(x)$ in $K'$.

\subsection{Proof of Proposition~\ref{P:final char 2 conditions}}

By Lemma~\ref{L:commutator q=2}, it suffices to prove that $\beta\circ \rho_\phi\colon \Gal_F \to \widehat{A}^\times \times \{\pm 1\} \times \{\pm 1\}$ is surjective.

Take any $i\in \FF_2$.  With $\gamma_i$ as in \S\ref{SS:maximal abelian quotient}, we let $L_i$ be the subfield of $F^\sep$ fixed by the kernel of the homomorphism $\gamma_i \circ \rho_{\phi} \colon \Gal_F \to \{\pm 1\}$.  

\begin{lemma} \label{L:explicit quadratic resolvent}
We have $L_i=F(\alpha)$ with $\alpha$ a root of the polynomial \[x^2-x + j_\phi/(t+i)^2+1 \in F[x].\]
\end{lemma}
\begin{proof}
We have $\phi_t=t+a_1\tau + a_2 \tau^2$ for some $a_1\in F$ and $a_2\in F^\times$ and hence $\phi_{t+i}=(t+i)+a_1 \tau + a_2 \tau^2$.  We have $\phi[\lambda_i] \cong \FF_2^2$, so $\phi[\lambda_i] =\{0,r_1,r_2,r_3\}$ for distinct nonzero $r_1,r_2,r_3 \in F^\sep$.  The values $r_1,r_2,r_3$ are the distinct roots in $F^\sep$ of the polynomial 
\[
f(x):=x^3+ (a_1/a_2) x + (t+i)/a_2 = a_2^{-1}x^{-1}( (t+i)x + a_1 x^2 + a_2 x^4) \in F[x].
\]  
We have a character $\varepsilon \colon \Gal_F\to \{\pm 1\}$ corresponding to $f(x)$ as in \S\ref{SS:cubic}.

With respect to a basis of $\phi[\lambda_i]\cong \FF_2^2$, the action of the group $\GL_2(\FF_2)$ on $\{r_1,r_2,r_3\}$ is faithful and transitive and induces an isomorphism $\GL_2(\FF_2)\xrightarrow{\sim} \mathfrak{S}_3$.  Using this, we find that $\varepsilon$ agrees with $\gamma_i\circ \rho_{\phi}$.  Therefore, $L_i$ is the subfield of $F^\sep$ fixed by the kernel of $\varepsilon$.  From \S\ref{SS:cubic}, we find that $L_i \subseteq F^\sep$ is the splitting field over $F$ of the polynomial
\[
R_2(x)=x^2-3\big(\tfrac{t+i}{a_2}\big) x + \big(\tfrac{a_1}{a_2}\big)^3+9\big(\tfrac{t+i}{a_2}\big)^2 \in F[x].
\]
Setting $y=a_2/(t+i) \, x$ and using that our field has characteristic $2$, we find that $L_i$ is the splitting field over $F$ of the polynomial $y^2-y + {a_1^3}/({a_2(t+i)^2}) + 1 = y^2-y + {j_\phi}/{(t+i)^2} +1$.
\end{proof}

\begin{lemma} \label{L:totally ramified and surjective}
The homomorphism 
\begin{align*} 
\beta'\colon \Gal_F \to \{\pm 1\} \times \{\pm 1\},\quad \sigma\mapsto (\gamma_0(\rho_{\phi}(\sigma)), \gamma_1(\rho_{\phi}(\sigma)))
\end{align*}
is surjective and is totally ramified at the place $\infty$ of $F$.
\end{lemma}
\begin{proof}
Let $L$ be the subfield of $F^\sep$ fixed by the kernel of $\beta'$.  We have $L=L_0L_1$.  For $i\in \FF_2$, Lemma~\ref{L:explicit quadratic resolvent} implies that $L_i=F(\alpha_i)$ with $\alpha_i$ a root of the polynomial $f_i(x):=x^2-x +j_\phi/(t+i)^2+1$. 

We claim that each $L_i/F$ is a quadratic extension that is totally ramified at the place $\infty$.    Since we are in characteristic $2$, the roots of $f_i(x)$ in $L$ are $\alpha_i$ and $\alpha_i+1$.     In particular, $\alpha_i(\alpha_i+1) = j_\phi/(t+i)^2+1$.   We have $v_\infty(j_\phi/(t+i)^2) = v_\infty(j_\phi) + 2 <0$, where we have used our assumption that $v_\infty(j_\phi)\leq -5$.  Therefore,
\[
v_\infty(\alpha_i(\alpha_i+1))= v_\infty(j_\phi/(t+i)^2+1)=v_\infty(j_\phi/(t+i)^2)=v_\infty(j_\phi) + 2.
\]
Since $v_\infty(j_\phi)$ is odd by assumption, we find that $v_\infty(\alpha_i(\alpha_i+1))$ is a negative odd integer.  After extending $v_\infty$ to a $\QQ$-valued valuation on $F^\sep$, we find that one of the rational numbers  $v_\infty(\alpha_i)$ or $v_\infty(\alpha_i+1)$ is negative and hence $v_\infty(\alpha_i)=v_\infty(\alpha_i+1)<0$.   Therefore, $2v_\infty(\alpha_i)=v_\infty(\alpha_i(\alpha_i+1))$ is an odd integer and hence $v_\infty(\alpha_i)\notin \ZZ$.  We deduce that $L_i=F(\alpha_i)$ is a nontrivial extension of $F$ that is ramified at $\infty$.   The claim follows since  $[L_i:F]\leq 2$.

Define $\alpha:=\alpha_0+\alpha_1$; it is a root of the polynomial
\begin{align} \label{E:f2}
x^2-x+(j_\phi/t^2+1)+ (j_\phi/(t+1)^2+1)= x^2-x + j_\phi/(t(t+1))^2.
\end{align}
We claim that $F(\alpha)$ is a quadratic extension of $F$ that totally ramified at the place $\infty$.   The roots of (\ref{E:f2}) are $\alpha$ and $\alpha+1$, so $v_\infty(\alpha(\alpha+1))=v_\infty(j_\phi/(t(t+1))^2)=v_\infty(j_\phi)+4$ From our assumptions on $v_\infty(j_\phi)$,  we deduce that $v_\infty(\alpha(\alpha+1))$ is a negative odd integer.  Therefore, $v_\infty(\alpha)=v_\infty(\alpha+1)$ and $2v_\infty(\alpha)$ is an odd integer.  We have $v_\infty(\alpha)\notin \ZZ$ and hence $F(\alpha)/F$ is a nontrivial extension ramified at $\infty$.  The claim follows since $[F(\alpha):F]\leq 2$.

We now show that $\beta'$ is surjective.  It suffices to show that $[L:F]=4$.  Since $L=L_0L_1$ with $[L_i:F]=2$, it suffices to show that $L_0 \neq L_1$.   If $L_0=L_1$, then for any $\sigma\in \Gal_F$, we have $\sigma(\alpha_i)=\alpha_i$ for both $i\in \FF_2$ or $\sigma(\alpha_i)=\alpha_i+1$ for both $i\in \FF_2$.   So if $L_0=L_1$, then $\alpha=\alpha_0+\alpha_1$ is fixed by $\Gal_F$ and hence $\alpha\in F$.   Since $F(\alpha)/F$ is a nontrivial extension, we deduce that $L_0\neq L_1$.   This completes the proof that $\beta'$ is surjective.

Suppose that $L/F$ is not totally ramified at $\infty$.   Since $\beta'$ induces an isomorphism $\Gal(L/F)\cong \{\pm 1\}\times \{\pm 1\}$, one of the three quadratic extensions of $F$ in $L$ must be unramified at $\infty$.  However, the three quadratic extensions of $F$ in $L$ are $L_0$, $L_1$ and $F(\alpha)$, and we have already shown that they are all ramified at $\infty$.  Therefore, $L/F$ is totally ramified at $\infty$.
\end{proof}

\begin{lemma} \label{L:tame ramification Carlitz}
The homomorphism $\det\circ \rho_\phi\colon \Gal_F \to \widehat{A}^\times$ is surjective and tamely ramified at the place $\infty$ of $F$.  
\end{lemma}
\begin{proof}
We have $\phi_t=t+a_1\tau+a_2\tau^2$ with $a_1\in F$ and $a_2\in F^\times$.   Let $\psi\colon A\to F\{\tau\}$ be the rank $1$ Drinfeld module for which $\psi_t=t-a_2\tau=t+a_2\tau$.  By Corollary 4.6 in \cite{MR1223025}, we have $\det \rho_{\phi} = \rho_{\psi}$.   

So it suffices to show that $\rho_\psi\colon \Gal_F \to \widehat{A}^\times$ is surjective and tamely ramified at $\infty$.   We have $a_2 \psi_t a_2^{-1} = t + \tau$.  So after replacing $\psi$ by an isomorphic Drinfeld module, we may assume that $\psi$ is the Carlitz module, i.e., $\psi_t=t+\tau$. 
For any nonzero ideal $\aA$ of $A$, the extension $F(\phi[\aA])/F$ is tamely ramified at $\infty$ and $\Gal(F(\phi[\aA])/F)\cong (A/\aA)^\times$, cf.~\cite[Theorems~2.3 and 3.1]{MR0330106}.  The lemma is now immediate.
\end{proof}

We need to show that 
\[
\beta\circ \rho_\phi\colon \Gal_F  \to  \widehat{A}^\times \times \{\pm 1\} \times \{\pm 1\}
\]
 is surjective.  Composing $\beta\circ \rho_\phi$ with the projection to $\widehat{A}^\times$ gives the homomorphism $\det\circ \rho_\phi$ which is surjective and tamely ramified at $\infty$ by Lemma~\ref{L:tame ramification Carlitz}.    Composing $\beta\circ \rho_\phi$ with the projection to $\{\pm 1\} \times \{\pm 1\}$ gives a homomorphism $\beta'$ which is surjective by Lemma~\ref{L:totally ramified and surjective}.  

Suppose $\beta\circ \rho_\phi$ is not surjective. Goursat's lemma (\cite[Lemma 5.2.1]{MR457455}) implies that there is a continuous and surjective homomorphism $\varphi\colon \Gal_F \to Q$, with $Q\neq 1$ a finite group, that factors through both $\det\circ \rho_\phi$ and $\beta'$.     The homomorphism $\varphi$ is tamely ramified at $\infty$ since it factors through $\det\circ \rho_\phi$.  However, since $\varphi$ factors through $\beta'$, Lemma~\ref{L:totally ramified and surjective} implies that  $\varphi$ is wildly ramified at $\infty$.  This gives a contradiction since $\varphi\neq 1$ and thus $\beta\circ \rho_\phi$ is surjective.

\subsection{Proof of Theorem~\ref{T:main new} when $q= 2$} \label{SS:main new proof -special}

We assume $q=2$.  Let $\calC$ be the set of $(a_1,a_2)\in A^2$ with $a_1 a_2\neq 0$ and $\deg(a_1)=\deg(a_2)-1$.  For any integer $d\geq 1$, we have
\[
|\calC(d)| = \sum_{i=1}^d (q-1)q^{i-1}\cdot (q-1)q^i=  (q-1)^2 q (q^{2d}-1)/(q^2-1).
\]
Therefore,
\[
\delta(\calC) = \lim_{d\to \infty} \frac{(q-1)^2 q (q^{2d}-1)/(q^2-1)}{q^{2(d+1)}} =  \frac{(q-1)^2 q}{q^2(q^2-1)} = \frac{1}{6}.
\]

Let $\calB$ be the set from \S\ref{S:sieving}; it has density $1$ by Theorem~\ref{T:almost main}.  We have $S_3 \supseteq \calB$ and hence $S_3$ has density $1$.

Take any $a\in \calB$.  Theorem~\ref{T:Gekeler}(\ref{T:Gekeler ii}) and $q=2$ implies that $\det( \rho_{\phi(a)}(\Gal_F))=\widehat{A}^\times$.  Since $\rho_{\phi(a)}(\Gal_F)\supseteq [\GL_2(\widehat{A}),\GL_2(\widehat{A})]$ and $\det( \rho_{\phi(a)}(\Gal_F))=\widehat{A}^\times$, Lemma~\ref{L:commutator q=2} implies that $[\GL_2(\widehat{A}): \rho_{\phi(a)}(\Gal_F)]$ divides $4$.  In particular, $a\in S_2$.  We have shown that $S_2\supseteq \calB$ and hence $S_2$ has density $1$.  

Now take any $a\in \calC\cap S_2$.   We have $j_{\phi(a)}=a_1^{q+1}/a_2=a_1^3/a_2$ and hence
\[
v_\infty(j_{\phi(a)}) = -v_\infty(a_2)+3v_\infty(a_1)= \deg(a_2)-3\deg(a_1)= -2 \deg(a_2)+3,
\] 
where the last equality uses that $\deg(a_1)=\deg(a_2)-1$.   Therefore, $v_\infty(j_{\phi(a)})$ is an odd integer and $v_\infty(j_{\phi(a)}) \leq -5$ when $\deg(a_2) \geq 4$.   So for any $a\in \calC\cap \calS_2$ with $\deg(a_2)\geq 4$, Proposition~\ref{P:final char 2 conditions} and $\rho_{\phi(a)}(\Gal_F)\supseteq [\GL_2(\widehat{A}),\GL_2(\widehat{A})]$ implies that $\rho_{\phi(a)}(\Gal_F)=\GL_2(\widehat{A})$.  The set $\calC\cap S_2$ has density $1/6$ and contains only finitely many $a$ with $\deg(a_2)<4$.  Therefore, $S_1$ has a subset with positive density.

\section{Proof of Theorem~\ref{T:example}}
\label{S:proof of examples}

\subsection{Proof of Theorem~\ref{T:example} when $q\neq 2$}

Fix a prime power $q>2$ and let $\phi\colon A\to F\{\tau\}$ be the Drinfeld $A$-module for which 
\[
\phi_t = t + \tau -t^{q-1} \tau^2.
\]
We will show that $\rho_\phi(\Gal_F)=\GL_2(\widehat{A})$.     

Define the prime ideal $\p:=(t)$ of $A$ and let $I_\p$ be an inertia subgroup of $\Gal_F$ at $\p$.  Observe that $\p$ is the only nonzero prime ideal of $A$ for which $\phi$ has bad reduction.   We have $j_\phi = -1/t^{q-1}$ and hence $v_\p(j_\phi)=-(q-1)$.   In particular, $\gcd(v_\p(j_\phi),q)=1$.  
 
 \begin{lemma} \label{L:example det unramified}
 For any nonzero ideal $\aA$ of $A$, the character $\det\circ \bbar\rho_{\phi,\aA}\colon \Gal_F \to (A/\aA)^\times$ is surjective and is unramified at all nonzero prime ideals $\q\nmid \aA$ of $A$.
 \end{lemma}
 \begin{proof}
 Let $\psi\colon A\to F\{\tau\}$ be the rank $1$ Drinfeld module for which $\psi_t=t-(-t^{q-1})\tau=t+t^{q-1} \tau$.  By Corollary 4.6 in \cite{MR1223025}, we have $\det \rho_{\phi} = \rho_{\psi}$ and hence also $\det \bbar\rho_{\phi,\aA} = \bbar\rho_{\psi.\aA}$.  We have $t \psi_t t^{-1} = t+ \tau$ and hence $\psi$ is isomorphic to the Carlitz module.  The lemma is now an immediate consequence of \cite[Proposition~2.2 and Theorem~2.3]{MR0330106}.
\end{proof}
 
 \begin{lemma} \label{L:example irreducibility}
For any nonzero prime ideal $\lambda$ of $A$, $\bbar\rho_{\phi,\lambda}$ is irreducible.
 \end{lemma}
 \begin{proof}
We will prove the lemma by contradiction. Suppose that $\bbar\rho_{\phi,\lambda}$ is reducible for some $\lambda$.    
 
First suppose that $\lambda=\p=(t)$.  For each nonzero $c\in \FF_q$, $\phi$ has good reduction at $(t-c)$.  We have $P_{\phi,(t-c)}(x)= x^2-x+t-c$ by Lemma~\ref{L:traces for degree 1 primes new}(\ref{L:traces for degree 1 primes new i}).  Therefore, $\bbar\rho_{\phi,\p}(\Gal_F) \subseteq \GL_2(\FF_\p)$ contains an element whose characteristic polynomial is $x^2-x - c\in \FF_\p[x]=\FF_q[x]$.  Since $\bbar\rho_{\phi,\lambda}$ is reducible, we find that the polynomial $x^2-x - c$ in $\FF_q[x]$ is reducible for all $c\in \FF_q$.   When $q$ is even, this is impossible since $\FF_q$ has a quadratic extension which must be given by a polynomial of the form $x^2-x - c\in \FF_q[x]$.  When $q$ is odd, this is also impossible since otherwise $(-1)^2-4(-c)=1+4c$ would be a square in $\FF_q$ for all $c\in \FF_q$.  Therefore, $\lambda\neq \p$.

After conjugating $\bbar\rho_{\phi,\lambda}$, we may assume that (\ref{E:upper triangular chis}) holds with characters $\chi_1,\chi_2\colon \Gal_F \to \FF_\lambda^\times$.    Since $\phi$ has good reduction away from $\p$, $\chi_1$ and $\chi_2$ are unramified at all nonzero prime ideals of $A$ except perhaps $\p$ and $\lambda$.  Proposition~\ref{P:main Tate new}(\ref{P:main Tate new i}), with $\aA:=\lambda$ and using $\lambda\neq \p$, implies that one of the characters $\chi_1$ or $\chi_2$ is unramified at $\p$.   Since $\det\circ \bbar\rho_{\phi,\lambda}=\chi_1\chi_2$ is unramified at $\p$ by Lemma~\ref{L:example det unramified}, we deduce that $\chi_1$ and $\chi_2$ are both unramified at $\p$.   Since $\bbar\rho_{\phi,\lambda}$ is reducible and $\phi$ has good reduction at $\lambda$, Proposition~\ref{P:good inertia} implies that $\chi_1$ or $\chi_2$ is unramified at $\lambda$.

We have verified that $\chi_1$ and $\chi_2$ satisfy (\ref{L:basic unramified chi new i}) and (\ref{L:basic unramified chi new ii}) of Lemma~\ref{L:basic unramified chi new} with $n:=1$.  By Lemma~\ref{L:riemann-hurwitz applied}, there is a $\zeta\in \FF_\lambda^\times$ such that $P_{\phi,\q}(\zeta^{\deg \q})= 0$ holds in $\FF_\lambda$ for all nonzero prime ideals $\q\notin \{\p,\lambda\}$ of $A$.

Assume that $q>3$ or $\deg \lambda >1$.  Then there are distinct nonzero $c_1,c_2\in \FF_q$ with $\lambda \in \{(t-c_1),(t-c_2)\}$.  Using Lemma~\ref{L:traces for degree 1 primes new}(\ref{L:traces for degree 1 primes new i}), we have 
\[
c_2-c_1=P_{\phi,(t-c_1)}(\zeta)-P_{\phi,(t-c_2)}(\zeta) \equiv  0 + 0 \equiv 0 \pmod{\lambda}
\]
which contradicts that $c_1$ and $c_2$ are distinct elements of $\FF_q$.

It remains to consider the case where $q=3$ and $\deg \lambda=1$.  In particular, $\lambda=(t-b)$ with $b\in \FF_3^\times$ and $\FF_\lambda=\FF_3$.  Define the prime ideal 
$\q:=(t^2 + t + 2)$ of $A$.   We have $\zeta^{\deg \q}=1$ since $|\FF_\lambda^\times|=2=\deg \q$.  Therefore, $P_{\phi,\q}(1) \equiv 0 \pmod{\lambda}$.   We have $P_{\phi,\q}(x)=x^2 + 2x + t^2 + t + 2$ by Lemma~\ref{L:traces for degree 1 primes new}(\ref{L:traces for degree 1 primes new ii}).  Therefore,
\[
0\equiv P_{\phi,\q}(1) \equiv 1^2 + 2\cdot 1 + b^2 + b + 2 \equiv b^2 + b + 2 \pmod{\lambda}
\]
and hence $b^2 + b + 2=0$ since $b\in \FF_3$.  However, this is a contradiction since $x^2+x+2$ is irreducible in $\FF_3[x]$.
\end{proof}
 
\begin{lemma}
\label{L:example lambda-adic surjectivity}
For any nonzero prime ideal $\lambda$ of $A$, we have $\rho_{\phi,\lambda}(\Gal_F)=\GL_2(A_\lambda)$.
\end{lemma}
\begin{proof}
Take any nonzero prime ideal $\lambda$ of $A$.   Set $G:=\rho_{\phi,\lambda}(\Gal_F) \subseteq \GL_2(A_\lambda)$.   Using Lemma~\ref{L:example det unramified} with $\aA=\lambda^i$ and $i\geq 1$, we find that $\det(G)=A_\lambda^\times$. 
 
Since $\gcd(v_\p(j_\phi),q)=1$, Proposition~\ref{P:main Tate new}(\ref{P:main Tate new iii}) implies that $\bbar\rho_{\phi,\lambda}(I_\p)$, and hence also $\bbar\rho_{\phi,\lambda}(\Gal_F)$, contains a subgroup of cardinality $N(\lambda)$.  The group $\bbar\rho_{\phi,\lambda}(\Gal_F)$ acts irreducibly on $\FF_\lambda^2$ by Lemma~\ref{L:example irreducibility}.  Proposition~\ref{P:group contains SL2 condition} implies that $\bbar\rho_{\phi,\lambda}(\Gal_F)\supseteq \SL_2(\FF_\lambda)$. Since $\det(G)=A_\lambda^\times$, we deduce that the image of $G$ modulo $\lambda$ is $\GL_2(\FF_\lambda)$.
 
Fix a generator $\pi$ of the ideal $\lambda$.   Let $P$ be a $p$-Sylow subgroup of $\bbar\rho_{\phi,\lambda^2}(I_\p)$, where $p$ is the prime dividing $q$.  Proposition~\ref{P:main Tate new}(\ref{P:main Tate new ii}) and $\gcd(v(j_\phi),q)=1$ implies that the cardinality of $P$ is divisible by $N(\lambda)^2$.   Since a $p$-Sylow subgroup of $\GL_2(\FF_\lambda)$ has cardinality $N(\lambda)$, we deduce that there is a $g\in P$ such that $g\equiv I + \pi B \pmod{\lambda^2}$ with $B\in M_2(A_\lambda)$ satisfying $B\not\equiv 0 \pmod{\lambda}$.     After conjugating our representation $\bbar\rho_{\phi,\lambda^2}$, we may assume by Proposition~\ref{P:main Tate new}(\ref{P:main Tate new ii}) that 
\[
\bbar\rho_{\phi,\lambda^2}(I_\p)\subseteq \left\{ \left(\begin{smallmatrix} a & b \\0 & c\end{smallmatrix}\right) : a\in (A/\lambda^2)^\times, b\in A/\lambda^2, c\in \FF_q^\times \right\}.
\]
Therefore, $P \subseteq \{  \left(\begin{smallmatrix} a & b \\0 & 1\end{smallmatrix}\right) : a\in (A/\lambda^2)^\times \text{ with }a\equiv 1 \pmod{\lambda}, b\in A/\lambda^2 \}$ and hence $B$ modulo $\lambda$ is of the form $ \left(\begin{smallmatrix} * & * \\0 & 0\end{smallmatrix}\right)$.  Since $B\not\equiv 0 \pmod{\lambda}$, we find that $B$ modulo $\lambda$ is not a scalar matrix.

We have now verified the conditions of Proposition~\ref{P:full GL2R criterion} with $q>2$ and hence $G=\GL_2(A_\lambda)$.
\end{proof}
 
Set $G:=\rho_{\phi}(\Gal_F)$.   We have $\det(G)=\widehat{A}^\times$ by Lemma~\ref{L:example det unramified}, so it remains to prove that $\SL_2(\widehat{A})$ is a subgroup of $G$.  For each nonzero prime ideal $\lambda$ of $A$, $\SL_2(A_\lambda)$ is a subgroup of $G_\lambda= \rho_{\phi,\lambda}(\Gal_F)$ by Lemma~\ref{L:example lambda-adic surjectivity}.   Take any two distinct nonzero prime ideals $\lambda_1$ and $\lambda_2$ of $A$.  Proposition~\ref{P:main Tate new}(\ref{P:main Tate new iii}) implies that $\bbar\rho_{\phi,\lambda_1\lambda_2}(I_\p)$ has a subgroup of cardinality $N(\lambda_1)N(\lambda_2)$.   In particular, $G$ modulo $\lambda_1\lambda_2$ has a subgroup of cardinality $N(\lambda_1)N(\lambda_2)$.
 Using Theorem~\ref{T:criterion for computing commutator} and $q>2$, we deduce that the commutator subgroup of $G$ is $\SL_2(\widehat{A})$ and hence $G\supseteq \SL_2(\widehat{A})$ as desired.

 \subsection{Proof of Theorem~\ref{T:example} when $q=2$}
 
 With $q=2$, let $\phi\colon A\to F\{\tau\}$ be the Drinfeld $A$-module for which 
\[
\phi_t=t+t^3\tau +(t^2+t+1)\tau^2.
\]  
We will show that $\rho_\phi(\Gal_F)=\GL_2(\widehat{A})$.    Since $q=2$, we have $\det(\rho_{\phi}(\Gal_F))=\widehat{A}^\times$ by Theorem~\ref{T:Gekeler}(\ref{T:Gekeler i}).

 Define the ideal $\p:=(t^2+t+1)$ of $A$.  Observe that $\p$ is the only nonzero prime ideal of $A$ for which $\phi$ has bad reduction.  Moreover, $\phi$ has stable reduction of rank $1$ at $\p$. We let $I_\p$ be an inertia subgroup of $\Gal_F$ for the prime $\p$.   We have $j_\phi = (t^3)^{q+1}/(t^2+t+1)=t^9/(t^2+t+1)$ and hence $v_\p(j_\phi)=-1$. 
 
\begin{lemma} \label{L:last q=2 example}
The homomorphism $\Gal_F \to \GL_2(\widehat{A})/[ \GL_2(\widehat{A}), \GL_2(\widehat{A})]$ obtained by composing $\rho_\phi$ with the obvious quotient map is surjective.
\end{lemma}
\begin{proof}
This follows from Proposition~\ref{P:final char 2 conditions} since $v_\infty(j_\phi)=-7$.
\end{proof}

 \begin{lemma} \label{L:example q=2 image mod}
We have $\bbar\rho_{\phi,\lambda}(\Gal_F)=\GL_2(\FF_\lambda)$ for all nonzero prime ideals $\lambda$ of $A$.
\end{lemma}
\begin{proof}
Assume that $\bbar\rho_{\phi,\lambda}(\Gal_F)\neq \GL_2(\FF_\lambda)$ for some $\lambda$.
Proposition~\ref{P:main Tate new}(\ref{P:main Tate new iii}) implies that $\bbar\rho_{\phi,\lambda}(I_\p)$ contains a subgroup of order $N(\lambda)$.   The representation $\bbar\rho_{\phi,\lambda}$ is thus reducible by Proposition~\ref{P:group contains SL2 condition}.  After conjugating $\bbar\rho_{\phi,\lambda}$, we may assume that 
\[
\bbar\rho_{\phi,\lambda}(\sigma) = \left(\begin{smallmatrix} \chi_1(\sigma) & * \\ 0 & \chi_2(\sigma) \end{smallmatrix}\right)
\]
for all $\sigma\in \Gal_F$, where $\chi_1,\chi_2\colon \Gal_F \to \FF_\lambda^\times$ are characters.

First assume that $\deg \lambda=1$ and hence $\lambda=(t+i)$ for some $i\in \FF_2$.  The nonzero elements of $\phi[\lambda]$ are the roots of the separable polynomial
\[
Q(x):=(t+i)+t^3x +(t^2+t+1)x^3 \in A[x].
\]  We have $\chi_1=1$ since $\FF_\lambda=\FF_2^\times$ and hence there is a nonzero element of $\phi[\lambda]$ lying in $F$.    In particular, $Q(x)$ has a root in $F$.  Therefore, the image of $Q(x)$ in $\FF_\q[x]$ has a root in $\FF_\q$ for all nonzero prime ideals $\q\neq (t^2+t+1)$ of $A$.  However, a computation shows that $Q(x)$ is irreducible modulo $\q$ for some prime $\q\in \{(t+1),(t^3 + t + 1)\}$.

Therefore, $\deg \lambda > 1$.    Since $\phi$ has good reduction away from $\p$, we find that $\chi_1$ and $\chi_2$ are unramified at all nonzero prime ideals of $A$ except perhaps $\p$ and $\lambda$.   When $\lambda=\p$, Proposition~\ref{P:main Tate new}(\ref{P:main Tate new i}) and $\FF_q^\times=\{1\}$ imply that $\chi_1$ or $\chi_2$ is unramified at $\p$.   When $\lambda\neq \p$, Proposition~\ref{P:main Tate new}(\ref{P:main Tate new i}) and $\FF_q^\times=\{1\}$ imply that both $\chi_1$ and $\chi_2$ are unramified at $\p$.    When $\lambda\neq \p$, $\phi$ has good reduction at $\lambda$ and so Proposition~\ref{P:good inertia}, with the reducibility of $\bbar\rho_{\phi,\lambda}$, implies that $\chi_1$ or $\chi_2$ is unramified at $\lambda$.   In particular, we have verified that parts (\ref{L:basic unramified chi new i}) and (\ref{L:basic unramified chi new ii}) of Lemma~\ref{L:basic unramified chi new} hold with $n:=1$.  

By Lemma~\ref{L:riemann-hurwitz applied} with $n=1$, there is a $\zeta\in \FF_\lambda^\times$ such that $P_{\phi,\q}(\zeta^{\deg \q})=0$ in $\FF_\lambda$ for all nonzero prime ideals $\q\neq \lambda$ of $A$ for which $\phi$ has good reduction.  Since $\deg \lambda >1$, we find that $P_{\phi,(t)}(x)$ and $P_{\phi,(t+1)}(x)$ have a common root modulo $\lambda$. Therefore,  the resultant $r\in A$ of $P_{\phi,(t)}(x)$ and $P_{\phi,(t+1)}(x)$ is divisible by $\lambda$.   The polynomials $P_{\phi,(t)}(x)$ and $P_{\phi,(t+1)}(x)$ where computed in Lemma~\ref{L:traces for degree 1 primes new}(\ref{L:traces for degree 1 primes new iii}) and one finds that $r=t+1$.   Therefore, $r=t+1 \equiv 0 \pmod{\lambda}$ which is a contradiction since $\deg \lambda>1$.
\end{proof}

\begin{lemma} \label{L:example q=2 adic surj}
We have $\rho_{\phi,\lambda}(\Gal_F)=\GL_2(A_\lambda)$ for all nonzero prime ideals $\lambda$ of $A$.
\end{lemma}
\begin{proof}
Set $G:=\rho_{\phi,\lambda}(\Gal_F)$.   We have $\det(G)=A_\lambda^\times$ since $\det(\rho_{\phi}(\Gal_F))=\widehat{A}^\times$.  The image of $G$ modulo $\lambda$ is equal to $\GL_2(\FF_\lambda)$ by Lemma~\ref{L:example q=2 image mod}.  

By Proposition~\ref{P:main Tate new}(\ref{P:main Tate new i}) and (\ref{P:main Tate new ii}), $\bbar\rho_{\phi,\lambda^2}(I_\p)$ has a subgroup of cardinality $N(\lambda)^2$ that is conjugate in $\GL_2(A/\lambda^2)$ to a subgroup of $\left\{ \left(\begin{smallmatrix} a & b \\0 & 1\end{smallmatrix}\right) : a\in (A/\lambda^2)^\times,   b\in A/\lambda^2\right\}$.  So with a fixed uniformizer $\pi$ of $A_\lambda$, we find that $G$ contains a matrix of the form $I+\pi B$ with $B \in M_2(A_\lambda)$ and $B$ modulo $\lambda$ not a scalar matrix.   When $\deg \lambda > 1$ and hence $|\FF_\lambda|>2$, Proposition~\ref{P:full GL2R criterion} implies that $G=\GL_2(A_\lambda)$.

We now assume that $\deg \lambda=1$ and hence $\FF_\lambda=\FF_2$.  Since $\lambda\neq \p$ and $\gcd(v(j_\phi),q)=1$, Proposition~\ref{P:main Tate new}(\ref{P:main Tate new iii}) implies that for any $i\geq 1$, $\bbar\rho_{\phi,\lambda^i}(I_\p)$ contains a subgroup that is conjugate in $\GL_2(A/\lambda^i)$ to $\left\{ \left(\begin{smallmatrix} 1 & b \\0 & 1\end{smallmatrix}\right) : b\in A/\lambda^i\right\}$.  Using that $G$ is closed, we find that $G$ has an element that is conjugate in $\GL_2(A_\lambda)$ to a matrix of the form $\left(\begin{smallmatrix} 1 & b \\0 & 1\end{smallmatrix}\right)$ with $b\not\equiv 0 \pmod{\lambda}$.  In particular, $G$ has an element with determinant $1$ whose image modulo $\lambda$ has order $2$.  

We will now show that $\bbar G:= \bbar\rho_{\phi,\lambda^2}(\Gal_F)$ is equal to $\GL_2(A/\lambda^2)$.  Let $S$ be the subgroup of $\SL_2(A/\lambda^2)$ consisting of matrices whose image modulo $\lambda$ is equal to $[\GL_2(\FF_\lambda),\GL_2(\FF_\lambda)]$.  Since $\FF_\lambda=\FF_2$, we have $|S|=2^3\cdot 3$ and $[\GL_2(A/\lambda^2): S]=2^2$.  The quotient $\GL_2(\FF_\lambda)/S$ is abelian and hence the quotient map $\bbar G \to \GL_2(\FF_\lambda)/S$ is surjective by Proposition~\ref{P:final char 2 conditions}.   In particular, $[\GL_2(A/\lambda^2): \bbar G] = [S: S\cap \bbar G]$.   We know that $\bbar G$ contains a matrix $g$ that is conjugate in $\GL_2(A/\lambda)$ to some $\left(\begin{smallmatrix} 1 & b \\0 & 1\end{smallmatrix}\right)$ with $b\not\equiv 0 \pmod{\lambda}$.   Also $g$ is not stable under conjugation by $\bbar{G}$ since the image of $\bbar{G}$ modulo $\lambda$ is $\GL_2(\FF_\lambda)$.  Therefore, $|S\cap \bbar G|$ is divisible by $4$.   The group $S\cap \bbar G$ contains an element of order $3$ since the image of $\bbar G$ modulo $\lambda$ is $\GL_2(\FF_\lambda)=\GL_2(\FF_2)$ and $[\GL_2(A/\lambda^2): S]=2^2$.  Since $|S|=2^3\cdot 3$, we find that $[\GL_2(A/\lambda^2): \bbar G] = [S: S\cap \bbar G]$ is equal to $1$ or $2$.  If $[\GL_2(A/\lambda^2): \bbar G]=2$, then $\bbar G$ is a normal subgroup of $\GL_2(A/\lambda^2)$ with nontrivial abelian quotient; this would contradict Proposition~\ref{P:final char 2 conditions}.  Therefore, we have index $1$, i.e., $\bbar G=\GL_2(A/\lambda^2)$.

We have now verified the conditions need to apply Proposition~\ref{P:full GL2R criterion} with $|\FF_\lambda|=2$ to show that $G=\GL_2(A_\lambda)$.
\end{proof}

\begin{lemma} \label{L:example q=2 contains commutator subgroup}
We have $\rho_\phi(\Gal_F)\supseteq [\GL_2(\widehat{A}),\GL_2(\widehat{A})]$.
\end{lemma}
\begin{proof}
Define $G:=\rho_\phi(\Gal_F)$; it is a closed subgroup of $\GL_2(\widehat{A})$.  We have already observed that $\det(G)=\widehat{A}^\times$.  The group $G_\lambda=\rho_{\phi,\lambda}(\Gal_F)$ is equal to $\GL_2(A_\lambda)$ for all nonzero prime ideals $\lambda$ of $A$ by Lemma~\ref{L:example q=2 adic surj}.

Take any two distinct nonzero prime ideals $\lambda_1$ and $\lambda_2$ of $A$ with $\deg \lambda_1=\deg \lambda_2$.     By Proposition~\ref{P:main Tate new}(\ref{P:main Tate new iii}), $\bbar\rho_{\phi,\lambda_1\lambda_2}(I_\p)\subseteq \bbar\rho_{\phi,\lambda_1\lambda_2}(\Gal_F)$ has a subgroup of cardinality $N(\lambda_1)N(\lambda_2)=N(\lambda_1)^2$.   

Now suppose that $\deg \lambda_1=\deg \lambda_2=1$ and hence $\p\notin \{\lambda_1,\lambda_2\}$.  For any integer $i\geq 1$, Proposition~\ref{P:main Tate new}(\ref{P:main Tate new iii}) implies that $\bbar\rho_{\phi,\lambda_1^i \lambda_2^i}(I_\p)\subseteq \bbar\rho_{\phi,\lambda_1^i \lambda_2^i}(\Gal_F)$ contains a subgroup conjugate in $\GL_2(A/(\lambda_1^i \lambda_2^i))$ to $\left\{ \left(\begin{smallmatrix}1 & b \\0 & 1\end{smallmatrix}\right) : b \in A/(\lambda_1^i\lambda_2^i) \right\}.$  Therefore, the closed group $G_{\lambda_1\lambda_2}=\rho_{\phi,\lambda_1\lambda_2}(\Gal_F)$ contains a subgroup that is conjugate in $\GL_2(A_{\lambda_1\lambda_2})$ to $\left\{ \left(\begin{smallmatrix}1 & b \\0 & 1\end{smallmatrix}\right) : b \in A_{\lambda_1\lambda_2} \right\}$. We have $\det(\rho_{\phi,\lambda_1\lambda_2}(\Gal_F))=A_{\lambda_1\lambda_2}^\times$ since $\det(\rho_\phi(\Gal_F))=\widehat{A}^\times$.

We have verified the conditions of Theorem~\ref{T:criterion for computing commutator} for $G$ and hence $G\supseteq [G,G]=[\GL_2(\widehat{A}),\GL_2(\widehat{A})]$.
\end{proof}
 
Lemmas~\ref{L:last q=2 example} and \ref{L:example q=2 contains commutator subgroup} now imply that $\rho_\phi(\Gal_F)=\GL_2(\widehat{A})$.

\DefineSimpleKey{bib}{primaryclass}{}
\DefineSimpleKey{bib}{archiveprefix}{}

\BibSpec{arXiv}{%
  +{}{\PrintAuthors}{author}
  +{,}{ \textit}{title}
  +{}{ \parenthesize}{date}
  +{,}{ arXiv }{eprint}
  +{,}{ primary class }{primaryclass}
}

\bibliographystyle{plain}


\end{document}